\documentclass[final,preprint,3p,sort&compress]{elsarticle}
\journal{{\tt arXiv.org}}

\usepackage{standalone}
\usepackage{fonts}
\usepackage{amssymb}
\usepackage{amsmath}
\usepackage{amsthm}
\usepackage{enumerate}
\usepackage{amstext}
\usepackage{graphicx}
\usepackage{hhline}
\usepackage{psfrag}
\usepackage{float}
\usepackage[dvips]{epsfig}
\usepackage{rotating}
\usepackage{graphics}
\usepackage{algorithm}
\usepackage{color}
\usepackage{caption}
\usepackage{subcaption}
\usepackage{placeins}
\definecolor{sepia}         {cmyk}{0   , 0.83, 1   , 0.70}
\definecolor{maroon}        {cmyk}{0   , 0.87, 0.68, 0.32}
\usepackage[plainpages=false,pdfpagelabels,bookmarksnumbered,
colorlinks=true,%
linkcolor=sepia,%
citecolor=sepia,%
filecolor=maroon,%
pagecolor=red,%
urlcolor=sepia]{hyperref}

\makeatletter\@addtoreset{equation}{section}\makeatother

\allowdisplaybreaks

\usepackage[]{hyperref} 
        
\usepackage{pgfplots}
\pgfplotsset{compat=1.12}
\usepgfplotslibrary{groupplots} 
\newcommand{\externaltikz}[2]{\includegraphics{Externals/#1}}

\usepackage{relinput}

\newtheorem{thm}{Theorem}
\newtheorem{example}{Example}	
\newtheorem{remark}{Remark}	
\newtheorem{definition}{Definition}	

\newcounter{tikzsubfigcounter}[figure]
\renewcommand{\thetikzsubfigcounter}{\the\numexpr\value{figure}+1\relax\alph{tikzsubfigcounter}}

\newcounter{tikzsubfigcounterinvisible}[figure]
\renewcommand{\thetikzsubfigcounterinvisible}{\the\numexpr\value{figure}+1\relax\alph{tikzsubfigcounterinvisible}}

\newcommand{\settikzlabel}[1]{ %
\refstepcounter{tikzsubfigcounterinvisible} \label{#1} 
}

\newlength{\figureheight}
\newlength{\figurewidth}
\newlength{\figureheightsave}
\newlength{\figurewidthsave}

\newcommand{\withfiguresize}[3]{%
	\setlength{\figureheightsave}{\figureheight}%
	\setlength{\figurewidthsave}{\figurewidth}%
	\setlength{\figureheight}{#1}%
	\setlength{\figurewidth}{#2}%
	#3%
	\setlength{\figureheight}{\figureheightsave}%
	\setlength{\figurewidth}{\figurewidthsave}%
}
\newcommand{\Sbb}{\mathbb{S}}

\newcommand{\tr}{\mbox{tr}}

\DeclareSymbolFont{msbm}{U}{msb}{m}{n}
\DeclareMathSymbol{\N}{\mathalpha}{msbm}{'116}
\DeclareMathSymbol{\R}{\mathalpha}{msbm}{'122}

\newcommand{\de}{\partial}
\newcommand{\dt}{\de_t}

\newcommand{\norm}[1]{\ensuremath{\Vert #1 \Vert}}
\newcommand{\US}{\mathbb{S}^2}
\newcommand{\intV}[1]{\ensuremath{\langle #1 \rangle}}

\newcommand{\pfrac}[2]{\ensuremath{\frac{\partial #1}{\partial #2}}}
\newcommand{\trans}[1][]{^{#1\top}}

\newcommand{\relerr}[2]{\ensuremath{e_{rel}(#1,#2)}}

\newcommand{\dx}{\partial_{x}}

\newcommand{\Flux}{\ensuremath{\bF}}
\newcommand{\Source}{\ensuremath{\bs}}

\newcommand{\PD}{\ensuremath{f}} % particle distribution
\newcommand{\FD}{\ensuremath{F}} % fibre distribution
\newcommand{\VF}{\ensuremath{Q}} % volume fraction

\newcommand{\basisf}{\pmb{a}}

\newcommand{\momf}{\pmb{u}}

\newcommand{\moments}[1][ ]{\ensuremath{\bu_{#1}}} %moment vector

\newcommand{\mpl}{\pmb{\alpha}}
\newcommand{\mplf}{\mpl}

\newcommand{\normal}[1]{\ensuremath{\hat{#1}}}

\newcommand{\figref}[1]{Figure~\ref{#1}}
\newcommand{\tabref}[1]{Table~\ref{#1}}

\newcommand{\trace}[1]{\ensuremath{\operatorname{trace}\left(#1\right)}}

\newcommand{\timevar}{\ensuremath{t}} %final time
\newcommand{\timevarref}{\ensuremath{\xi}} %final time

\newcommand{\dtstepsize}{\ensuremath{\Delta \timevar}}

\newcommand{\dxstepsize}{\ensuremath{\Delta x}}

\newcommand{\timeind}{\ensuremath{\kappa}}
\newcommand{\timepoint}[1]{\ensuremath{\timevar_{#1}}}
\newcommand{\cellind}{\ensuremath{j}}

\newcommand{\cellmean}[2][\cellind]{\ensuremath{\overline{#2}_{#1}}}

\newcommand{\momentsprojected}{\ensuremath{\moments[h]}}

\newcommand{\momentspv}[1]{\ensuremath{\moments[#1]}}

\newcommand{\momentscellmean}[1]{\ensuremath{\cellmean[#1]{\moments}}}

\newcommand{\momentslocal}[1]{\moments[#1]}
\newcommand{\momentslocallimited}[2][\limitervariable]{\moments[#2]^{#1}}

\newcommand{\Testfunction}[1][ ]{\ensuremath{v_{#1}}}

\newcommand{\numericalFlux}{\ensuremath{\widehat{\bF}}}
\newcommand{\viscosityconstant}{\ensuremath{C}}

\newcommand{\RKstages}{\ensuremath{s}}

\makeatletter
\DeclareFontFamily{U}{tipa}{}
\DeclareFontShape{U}{tipa}{m}{n}{<->tipa10}{}
\newcommand{\arc@char}{{\usefont{U}{tipa}{m}{n}\symbol{62}}}%

\newcommand{\arc}[1]{\mathpalette\arc@arc{#1}}

\newcommand{\arc@arc}[2]{%
  \sbox0{$\m@th#1#2$}%
  \vbox{
    \hbox{\resizebox{\wd0}{\height}{\arc@char}}
    \nointerlineskip
    \box0
  }%
}
\makeatother

\pgfplotscreateplotcyclelist{color parula}{% 
	royalblue!70,every mark/.append style={solid,line width = 0pt,fill=royalblue!60!black},mark=ball\\%
	goldenrod!50!yelloworange,every mark/.append style={solid,fill=goldenrod!30!black},mark=*\\%
	green,every mark/.append style={black,fill=yellowgreen},mark=diamond*\\%
	bluegreen,every mark/.append style={fill=black},mark=triangle*\\%	
	royalblue!70,densely dashed,every mark/.append style={solid,line width = 0pt,fill=royalblue!60!black},mark=ball\\%
	goldenrod!50!yelloworange,densely dashed,every mark/.append style={solid,black,fill=goldenrod},mark=diamond*\\%
	yellowgreen,densely dashed,every mark/.append style={solid,fill=yellowgreen!60!royalblue},mark=square*, mark size= 1pt\\%
	bluegreen,densely dashed,mark size = 1.5pt,every mark/.append style={solid,fill=white},mark=otimes*\\%,densely dashed	
	royalblue,densely dashed,mark size = 1.5pt,every mark/.append style={solid,fill=royalblue!30!black},mark=pentagon*\\%,densely dashed
	goldenrod!40!yelloworange,every mark/.append style={solid,fill=goldenrod!30!black},mark=|\\%		
	}
	
\usepackage[colorinlistoftodos,obeyFinal]{todonotes}
		
\begin{document}
	
	\markboth{G. Corbin, A. Hunt,  F. Schneider, A. Klar, C. Surulescu}{Glioma}
	
%	%%%%%%%%%%%%%%%%%%% Publisher's Area please ignore %%%%%%%%%%%%%%%%%%%%%%%
%	%
%	\catchline{}{}{}{}{}
%	%
%	%%%%%%%%%%%%%%%%%%%%%%%%%%%%%%%%%%%%%%%%%%%%%%%%%%%%%%%%%%%%%%%%%%%%%%%%%%
	
	\title{Higher-order models for glioma invasion: from a two-scale description to effective equations for mass density and momentum
		\thanks{This work was financially supported by BMBF in the project GlioMaTh.}}

	\author{G. Corbin, A. Hunt, A. Klar, F. Schneider, C. Surulescu}
	
	\address{Department of Mathematics, University of Kaiserslautern, \\ 
		P.O. Box 3049, 67653 Kaiserslautern, Germany\\
		{\it \{corbin,hunt,klar,schneider,surulescu\}@mathematik.uni-kl.de}}
	
	%\address{$\dagger$Fraunhofer ITWM Kaiserslautern, \\ 
	%67653 Kaiserslautern, Germany
	
	\begin{abstract}
		Starting from a two-scale description involving receptor binding dynamics and a kinetic transport equation for the evolution of the cell density function under velocity reorientations, 
		we deduce macroscopic models for glioma invasion featuring partial differential equations for the mass density and momentum of a population of glioma cells migrating through the anisotropic 
		brain tissue. The proposed first and higher order moment closure methods enable numerical simulations of the kinetic equation. Their performance is then compared to that of the  
		diffusion limit.  The approach allows for DTI-based, patient-specific predictions of the tumor extent and its dynamic behavior. 
	\end{abstract}
	
	\maketitle
	
	%\begin{history}
	%\received{(Day Month Year)}
	%\revised{(Day Month Year)}
	%\accepted{(Day Month Year)}
	%\comby{(xxxxxxxxxx)}
	%\end{history}

%	\keywords{Multiscale model, glioma invasion, kinetic transport equation, diffusion tensor imaging, macroscopic scaling, moment closure, reaction-diffusion-transport equations}
%	
%	\ccode{2010 AMS Subject Classification: 92C17, 92C50, 35Q92, 90B20, 35L60, 35L65} 

\noindent

% {\bf Key words.}

\section{Introduction}

\noindent
Malignant glioma make up about half of all primary brain tumors in adults. These rapidly growing tumors invade adjacent regions of the brain tissue and occur in all age groups, with a higher 
frequency in late adulthood. An exhaustive microscopic resection of the neoplasm is in general impossible due to their high proliferation rate and diffuse infiltration. This leads to substantial clinical 
challenges and high mortality of the affected patients. For glioblastoma multiforme (GBM), the most frequent and most aggressive type of these tumors, the median survival time amounts to 60 weeks 
in spite of the most advanced treatment methods, involving resection, radio- and chemotherapy \cite{wrenschetal2002}.

\noindent
The available non-invasive medical imaging techniques, including MRI and CT, only allow a macroscopic classification of the active and necrotic tumor areas and of the surrounding edema. Thus, the microscopic 
extent of the tumor cannot be visualized. Tumor growth, development, and invasion is, however, a complex multiscale phenomenon in which the macroscopic evolution of the tumor is regulated by processes on lower  
scales at the cellular and subcellular levels \cite{hanahan2011}. Particularly the issue of cell invasion into the tissue is of utmost relevance in the context of assessing the tumor margins. The latter 
are most 
often very diffuse and irregular \cite{coons,daumas-duportetal,gerstneretal2010} due to the highly infiltrative migration of glioma cells, which are believed to follow white matter tracts, thus using the 
anisotropy of brain tissue \cite{claes2007,d'abaco-kaye07,giese-etal96,giesewestphal1996}. As both the brain structure and the tumor evolution are patient-specific, a personalized approach to 
diagnosis and therapy is necessary. While surgical resection has to orient itself on the macroscopic tumor and chemotherapy is mainly systemic, the radiotherapeutic 
approach can greatly benefit from individual predictions of the macroscopic tumor extent based on microscopic information about brain tissue structure and therewith conditioned cell invasion. 
Mathematical models can provide a valuable tool for including such detailed information in the process, leading to enhanced forecasts of the tumor margins and thus to improvement of therapy planning.

\noindent
Existing mathematical models for glioma growth and invasion can be assigned to several classes: Discrete settings follow the evolution of individual cells on a 
lattice \cite{boettger-etal-12,hatzikirou-etal-12}, while hybrid models (see e.g. \cite{kim-roh-2013,tanakaetal2009}) combine discrete descriptions with continuous differential equations for densities, concentrations, 
 and/or volume fractions controlling the cell behavior. These approaches have the advantage of being able to include a high level of detail in the characterization of cell motions, but they also require a high computational 
effort and involve many parameters. Continuous approaches involve systems of various types of partial differential equations, sometimes coupled to ODEs, and allow to describe the dynamics of cells 
interacting with their surroundings. They are less detailed than their (semi)discrete counterparts, but are able to capture the essential features of the modeled processes, and are better suited for efficient 
numerical simulations. Most of them are directly set on the macroscopic scale, on which the tumor is observed, and involve reaction-diffusion equations for the density of glioma cells. Thereby, the influence 
of the anisotropic brain structure is included in the diffusion coefficients, which are assumed to be proportional to the water diffusion tensor assessed by diffusion tensor imaging (DTI) 
\cite{Jbabdi,Konukoglu2010,wang09}. More recent models involve phase field approaches, where the tumor cells, the healthy tissue, and nutrients are seen as phases interacting with each 
other \cite{colombo-etal}. Yet another approach \cite{PH13} uses a parabolic scaling to deduce the macroscopic description of the tumor density from kinetic transport equations for the cell density function depending on time, position, velocity. That approach has been further extended in \cite{EHS,EHKS14,EKS16,HS17} to 
include phenomena on the subcellular scale, by introducing so-called cell activity variables (see \cite{BBNS10a}) to the variable space. On the macrolevel, this leads to a 
reaction-diffusion-taxis equation for the tumor density. The new haptotaxis-like term arises from the cell receptor binding to 
the surrounding tissue. The resulting models are able to reproduce the fingering patterns of glioma mentioned in the medical literature. 

\noindent
In this paper we aim at providing a new perspective on the DTI-based, two-scale description of glioma evolution in \cite{EHS,EHKS14}, with a focus on some moment closure techniques 
allowing 
to numerically handle the mesoscopic kinetic transport equation. In the kinetic context, the distribution function for the tumor cells is 
a mesoscopic quantity depending not only on time and position, but also on the cell velocity and the activity variables mentioned above. Among other methods, these dependencies can be discretized by moment closures \cite{BruHol01,Schneider2014,Schneider2015c,Ritter2016,Chidyagwai2017,Klar2014,Lev84}, which transform the scalar, but high-dimensional transport equation into a hyperbolic system for moments of the cell density with respect to the velocity and activity variables. 

\noindent
The paper is structured as follows: Section \ref{section1} introduces the two-scale model and the corresponding kinetic equation on the mesoscopic level, which constitutes the starting point of 
the subsequent study. The model is further specified by an adequate choice of the turning kernels to describe cell reorientations in response to the interactions with the surrounding 
tissue; the latter including the DTI brain data. As the focus is on the study of the moment models, no proliferation or decay terms are considered. A parabolic scaling 
leads to the effective macroscopic equation for the tumor cell density, featuring myopic diffusion and the mentioned haptotaxis-like term. Sections \ref{1st-order-closure} and \ref{higer-order-closure} are 
dedicated to the derivation of first and higher order moment closures, respectively. The numerical simulations are presented in Section \ref{sec:Results}, which also provides a comparison between the 
performance of the different approaches considered in this work. Section \ref{sec:Remarks}, containing the discussion of the results, is followed by the Appendix, giving some details about the numerical schemes 
and their implementation along with a proof of the hyperbolicity of the obtained Kershaw moment system.

%%%%%%%%%%%%%%%%%%%%%%%%%%%%%%%%%%%%%%%%%%%%%%%%%%%%%%%%%%%%%%%%%%%%%%%%%%%%%%%%%%%%%
\section{The kinetic glioma model}\label{section1}
\subsection{From a two-scale description to the mesoscopic kinetic equation}
\label{section1a}
%%%%%%%%%%%%%%%%%%%%%%%%%%%%%%%%%%%%%%%%%%%%%%%%%%%%%%%%%%%%%%%%%%%%%%%%%%%%%%%%%%%%%
In this section we describe the kinetic system modeling glioma invasion developed in \cite{EHKS14,PH13}, from which we derive an approximating scalar kinetic equation. Let $x\in \R^n$, $t\in \R^+$, and $v \in V = c\,\Sbb^{n-1}$ denote the mechanical variables, i.e. a position vector, the time variable, and a velocity vector, respectively. Thereby, the speed of the cells is assumed constant.\\[1ex]
\noindent 
 Further, let $y$ represent the volume fraction of bound receptors on the cell membrane. Its dynamics are characterized via mass action kinetics of binding and detachment of receptors to the soluble 
 and unsoluble components of the cell environment. Here we consider only the latter kind of bindings, i.e. to the ligands on the extracellular matrix fibers, as we are primarily interested in the influence 
 of the specific individual brain structure on the glioma invasion. Translating the receptor binding kinetics into a differential equation leads to
 \begin{equation}\label{eq-dynint}
  \dot y=-(k^+\VF+k^-)y+k^+\VF,
 \end{equation}
where $\VF(t,x)$ represents the macroscopic volume fraction of tissue fibers, while $k^+$ and $k^-$ are positive constants denoting the binding and the detachment rate, respectively. Note that $y$ is a unitless 
quantity in the interval $(0,1)$ and, as in \cite{EHKS14}, we assume the total number of receptors to be conserved. The variable $y$ can be interpreted as a so-called biological variable, also referred to as an activity 
variable in the KTAP (kinetic theory of active particles) framework introduced by Bellomo et al. (see e.g., \cite{BBNS10a}). \\[1ex]
\noindent
Compared to cell motion, the reversible receptor binding to the extracellular matrix (ECM) fibers is a very fast process. Therefore, the corresponding dynamic equilibrates rapidly at its steady-state, which is uniquely given by 
$$y^\star=\frac{k^+\VF}{k^+\VF+k^-} := g(\VF),$$ 
for \eqref{eq-dynint}.
For the subsequent analysis we will only consider deviations from the steady-state, which will be small quantities $z:=y^*-y$ in the set $Z\subseteq (y^*-1,y^*)$. As in \cite{EHKS14}, 
consider the path of a single cell starting in $x_0$ and moving with velocity $v$ through the time-invariant density field $\VF(x)$. Then, with $x=x_0+vt$, we obtain that $z$ satisfies the equation 
\begin{equation*}
 \frac{dz}{dt}=-(k^+\VF(x)+k^-)z+ g'(\VF(x))v\cdot \nabla \VF(x).
\end{equation*}
With the notations $\alpha (\VF(x)):=k^+\VF(x)+k^-$ and $\beta (\VF(x)):=g'(\VF(x))\nabla \VF(x)$ this takes the form
\begin{equation*}
 \frac{dz}{dt}=-\alpha (\VF(x))z+\beta (\VF(x))\cdot v.
\end{equation*}
We use the mesoscopic cell density function $p(t,x,v,z)$ to describe the dynamics of glioma cells, and a velocity jump model whose integral operator models the velocity innovations in response to the 
tissue structure. The corresponding kinetic transport equation is written as follows (see \cite{EHKS14,EHS}):
\begin{equation}\label{pdeinitial}
 \partial_t p  + v \cdot \nabla_x p + \partial_z \left( \left( - \alpha z + v \cdot \beta \right) p \right)
= - \lambda (z)\left(p(v)- \int_V K(x,v,v^\prime)p(v^\prime) d v^\prime \right),
\end{equation}
on the domain
\begin{align*}
	\Omega_T \times \Omega_X \times V \times Z  \subseteq (0,T) \times [0,X]^n \times c\,\US \times [y^\ast-1,y^\ast],
\end{align*}
where $K(x,v,v')$ is the turning kernel carrying the tissue influence and $\lambda (z)$ is the cell turning rate. We assume that the kernel has the form $K(x,v,v')= \FD(x,v)$, where $\FD(x,v)$ represents the normalized directional distribution of tissue fibers \cite{EHKS14,EHS}, 
 i.e. $\int_{V} \FD(x,v)dv=1$. 
 For the turning rate we take 
$$\lambda (z):=\lambda_0 - \lambda_1 z$$
such that $\lambda _0,\lambda _1>0$ are constants. This choice corresponds to a turning rate increasing with the amount of receptors bound to the ECM, see \cite{EHKS14}.\\[1ex]
\noindent
Thus, \eqref{pdeinitial} takes the form
\begin{align}\label{KTE}
 \partial_t p  + v \cdot \nabla_x p &+ \partial_z \left( \left( - \alpha (\VF)z 
 + v \cdot \beta (\VF) \right) p \right)\\
 &=  - \lambda (z)\left(p(t,x,v,z)- \FD(x, v)\int_V p(t,x,v^\prime,z)  d v^\prime \right).\nonumber
 \end{align}

\noindent
Following the same lines as in \cite{EHKS14}, we derive a reduced kinetic problem from \eqref{KTE}, where the distribution function does no longer depend on the variable $z$. Integrating  
\eqref{KTE} with respect to $z$, assuming the solution to be compactly supported in the $(x, v, z)$-space, and defining $f(t,x,v): = \int _Z p(t,x,v,z )dz$, i.e. the integral of 
$p$ with respect to $z$,
we obtain
\begin{align}
\partial_t f  + v \cdot \nabla_x f 
= - \lambda_0 \left(f - \FD(x, v) \rho \right)+ \lambda_1 \left( f^z -\FD(x, v) \rho^z\right),\label{kineticsystem1}  
\end{align}
with $
\rho (t,x)  = \int_V f(t,x,v) dv$ denoting the macroscopic cell density, $
f^z(t,x,v) = \int_Z z p(t,x,v,z) dz$ representing the first moment with respect to $z$, 
and 
$
\rho^z (t,x)  = \int_V f^z(t,x,v) dv$ the associated density.
% 
% \begin{example}
% \begin{align}
% q (x,v)= \frac{n}{\vert S^{n-1} \vert tr D_W(x)} v^T D_W (x) v\\
% %\lambda_0(Q) \rightarrow 0 \; \mbox{for} \; Q \rightarrow 0\\
% \end{align}
% \end{example}

\noindent
An approximation for $f^z$ is obtained by the subsequent formal considerations.
Again following \cite{EHKS14}, we multiply \eqref{KTE} with $z$, integrate with respect to $z$, and neglect second order moments in $z$. This leads  to 
\begin{align*}
\partial_t f^z  + v \cdot \nabla_x f^z 
= - \lambda_0 \left(f^z - \FD(x,v) \rho^z \right)
- \alpha (Q)f^z + v \cdot \beta (Q)f.% \label{kineticsystem2}
\end{align*}
A quasi-stationarity assumption for $f^z$, i.e.
neglecting the transport terms in the $f^z$ equation, yields
\begin{align}\label{for-fz}
- \lambda_0 \left(f^z - \FD(x,v) \rho^z \right)
- \alpha (Q)f^z + v \cdot \beta (Q)f =0.
\end{align}
Integrating this with respect to $v$ and rearranging gives 
\begin{align*}
	\rho^z = \frac{1}{\alpha(\VF)} \beta(\VF) \cdot q,
\end{align*}
with
$
q= \int_V v f dv.
$ 
Then solving \eqref{for-fz} for $f^z$ yields
\begin{align}\label{mz}
f^z = \frac{1}{\lambda_0 + \alpha (Q)}
\beta (Q)\cdot \left( v  f + \frac{\lambda_0}{\alpha (Q)} \FD(x,v) q\right).
\end{align}
\noindent
Using this in \eqref{kineticsystem1} gives the still mesoscopic equation
\begin{align}
\label{kineticsystemreduced}
\partial_t f  + v \cdot \nabla_x f 
=- \lambda_0 \left(f - \FD( v) \rho \right) + \lambda_H \nabla_x \VF(v) \left(v  f - \FD ( v) q  \right),
\end{align}
with
\begin{align} 
\lambda_H(\VF(x)) &= \frac{\lambda_1}{\lambda_0} \hat \lambda_H (\VF (x))=  \frac{\lambda_1}{\lambda_0} \frac{1}{1 + \frac{\alpha (Q)}{\lambda_0}} g'(Q).
\label{turningrates}
\end{align} 
A non-dimensional form of \eqref{kineticsystemreduced} is
\begin{align}
\partial_t f + \frac{t_0 c}{x_0} \nabla_x \cdot (v f) &= - t_0 \lambda_0 \left(f - \FD(v) \rho\right)  + \frac{\lambda_1}{\lambda_0} \frac{t_0 c}{x_0}  \hat \lambda_H  \nabla_x \VF\left(f v - \FD(v) q \right),
\end{align}
on the domain
\begin{align*}
	\frac{\Omega_T}{t_0} \times \frac{\Omega_X}{x_0} \times \mathbb S^{n-1}.
\end{align*}
Identifying the Strouhal number $St = \frac{x_0}{t_0 c}$, the Knudsen number $Kn=\frac{1}{t_0 \lambda_0}$, and the ratio of  turning rate coefficients $\eta = \frac{\lambda_1}{\lambda_0}$ as the 
characteristic 
parameters, we write the above as
\begin{align}
\partial_t f + \frac{1}{St} \nabla_x \cdot (v f) &= -\frac{1}{Kn} \left(f - \FD(v) \rho\right)  + \frac{\eta}{St}  \hat \lambda_H  \nabla_x \VF \left(f v - \FD(v) q \right) .
\label{kineticsystemscaled}
\end{align}
\noindent
We consider a slightly generalized version with arbitrary kernels, i.e.
\begin{align}
\label{pdegeneralscaled}
\partial_t f  + \frac{1}{St} v \cdot \nabla_x f  &= L f := (\frac{1}{Kn}L_1 + \frac{\eta}{St} L_2) f,
\end{align}
and
\begin{align}
\label{Ldef}
L_i f =  \int_V \left( k_i(v,v^\prime) f(v^\prime) -  k_i( v^\prime,v) f (v) \right) d v^\prime,\quad i\in\{1,2\}.
\end{align}
Both turning operators conserve mass, i.e.
\begin{align}
	\int_V L_i(f)(v) dv = 0.
	\label{massconservingassumption}
\end{align}
 The first kernel as well as the combined kernel $k_1 + k_2$  are assumed strictly positive and bounded from above: 
\begin{alignat*}{3}
	0 &< k_{1,min} &&\leq k_1(v^\prime,v)        &&\leq k_{1,max}, \\
	0 &< k_{min}   &&\leq k_1(v^\prime,v) + k_2 (v^\prime,v) &&\leq k_{max}. 
\end{alignat*}
Additionally, the first kernel satisfies 
\begin{align*}
	\int_V  k_1(v^\prime,v)  dv^\prime = \kappa_1(x),	
\end{align*}
with some known function $\kappa_1$. Moreover, we assume that there is a probability distribution $F = F (x,v)$ fulfilling for each $x\in \Omega _X$ the detailed balance condition
\begin{align*}
k_1(v^\prime,v)F (x,v) = k_1(v, v^\prime)F (x,v^\prime),
\end{align*}
which we call the equilibrium distribution. 

\begin{example}[Kernels]\label{wahl-kernels}
	Note that \eqref{kineticsystemscaled} fits into the more general framework \eqref{pdegeneralscaled} by choosing the kernels $k_1$ and $k_2$ as 
	\begin{equation}
	\label{kernels-special}
	\begin{aligned}
	%k_1(v, v^\prime) &= \lambda_0 \FD(v) \\
	k_1(v, v^\prime) &=  \FD(x,v), \\
	k_2(v, v^\prime) &= - \hat \lambda_H \nabla_x \VF \cdot v^\prime \FD(x,v).
	\end{aligned}
	\end{equation}
	However, using 
	\begin{align*}
	k_2(v, v^\prime) = ( a v - b v^\prime ) F(v) \varphi(\nabla m )
	\end{align*}
	with constants $a $ and $b$, gives a flux limited chemotaxis type kernel \cite{CMPS04}. The cor\-res\-pon\-ding linear reorientation operator is
	\begin{align*}
	%L_2 f (v) = \left( a \left(\rho v F(v) - \int  v F dv   f\right)+ b \left(vf - \rho u F(v)\right) \right)  \varphi(\nabla m ).
	L_2 f (v) = \left( a \left(\rho v F(v) - \int _V v F(v) dv   f\right)+ b \left(vf - q F(v)\right) \right)  \varphi(\nabla m ).
		\end{align*}
	The flux-limiter $\varphi$ is chosen for example as 
	\begin{align*}
	\varphi(x) =  \frac{x}{\sqrt{1+\vert x \vert^2 }},
	\end{align*}
	compare \cite{Lev84,BKP17}. 	
	The chemoattractant concentration $m(t,x)$ is usually governed  by a diffusion equation 
	\begin{align*}
	\partial_t m - D_m \Delta_x m = \gamma \rho  - \delta m
	\end{align*}
	with a production proportional to the population density $\rho$ and an exponential decay with rate $\delta$. 
\end{example}

\subsection{From the mesoscopic equation to the diffusion limit}
For notational simplicity, we will drop the $t$ and $x$ dependency in the next sections and use the shorthand notation 
\begin{align*}
\intV{\cdot} := \int_V \cdot \ dv
\end{align*} for integrals over the velocity space.

\noindent
Defining the parabolic scaling parameter as $\epsilon := St$, we write \eqref{pdegeneralscaled} as
\begin{align}
	\partial_t f + \frac{1}{\epsilon} \nabla_x \cdot (v f) &= \frac{St^2}{Kn} \frac{1}{\epsilon^2} L_1 f  + \frac{\eta}{\epsilon}  L_2 f,
	\label{pdegeneraleps}
\end{align} which will converge to a diffusion equation for the cell density 
\begin{align*}
\rho(t,x) &:= \intV{f(t,x,v)}
\end{align*}
for $\epsilon \rightarrow 0$ while $\eta$ and the ratio $R:=\frac{St^2}{Kn}$ are fixed.
Following the works in \cite{LarKel75,BSS84,hillen2002hyperbolic}, we derive an equation for the cell density $\rho$ under the first-order symmetry assumption
\begin{align}
	\intV{ v F (v)} = 0.
	\label{symmetryassumption}
\end{align}

\noindent
We use the asymptotic expansion
\begin{align}
	f = f_0 + \epsilon f_1 + \mathcal{O}(\epsilon^2),
	\label{asymptoticexpansion}
\end{align}
which we plug into \eqref{pdegeneraleps}. Comparing coefficients of different orders in $\epsilon$ yields to zeroth order
\begin{align*}
L_1 f_0 = 0 \Longrightarrow f_0 (v)= \rho_0  F (v),
\end{align*}
and consequently to first order
\begin{align}
\label{firstorder}
v \cdot \nabla_x f_0 =  R L_1 f_1
+   \eta L_2 f_0 \quad \Leftrightarrow \quad L_1 f_1= \frac{1}{R} [v  \cdot \nabla_x (\rho_0 F(v)) 
- \eta \ \rho_0 L_2 F(v)].
\end{align}
With the notations 
\begin{subequations}\label{notation-Hs}
\begin{align}
L_1 H_1 (v)&=  -  \frac{1}{R} v F(v),\label{notations-Hsa}\\
L_1 H_2(v) &=  - \frac{\eta}{R} L_2 F (v)  +\frac{1}{R} v \cdot \nabla_x F(v)\label{notations-Hsb} 
\end{align}
\end{subequations}
we can write 
\begin{align}\label{L1-gl}
 L_1f_1=-L_1H_1(v)\cdot \nabla_x \rho _0+L_1H_2(v)\rho _0.
\end{align}
Since $\FD$ is first-order symmetric by \eqref{symmetryassumption}, we also have $\intV{ v \nabla_x F(v) } =0$. 
Together with mass conservation (see \eqref{massconservingassumption}) this ensures that the right hand side in \eqref{L1-gl} vanishes when integrated with respect to $v$, which -on an adequate function space- allows to solve \eqref{L1-gl} for $f_1$. We write, at this point formally,
\begin{align*}
f_1 =   -   H_1(v) \cdot \nabla_x \rho_0 + \rho_0 H_2(v).
\end{align*}
\noindent
Inserting the asymptotic expansion \eqref{asymptoticexpansion} into \eqref{pdegeneraleps} and integrating with respect to $v$ yields under assumptions \eqref{massconservingassumption} and \eqref{symmetryassumption}:
\begin{align*}
0 &= \partial_t \rho_0 + \frac{1}{\epsilon} \nabla_x \cdot \intV{vf_0 + \epsilon v f_1 + \mathcal{O}(\epsilon^2)}\\
  &= \partial_t \rho_0  +   \nabla_x \cdot \intV{  v f_1 } + \mathcal{O}(\epsilon)\\
  &=\partial_t \rho_0  -    \nabla_x \cdot \left( D \nabla_x     \rho_0  -   \Gamma   \rho_0  \right) +  \mathcal{O}(\epsilon),
\end{align*}
with 
\begin{align*}
D &:=   \intV { v \otimes  H_1(v)  },& \text{ (diffusion tensor)} \\
\Gamma &:= \intV{  v H_2(v)  }& \text{ (drift velocity)}.
\end{align*}

\begin{remark}
	For the special case of equation  (\ref{kineticsystemscaled}),  i.e. for the particular choice of kernels in \eqref{kernels-special}, we observe as in \cite{EHKS14} that if we consider the 
	weighted space\footnote{in which the inner product $(f,\FD(v))=\int_Vf(v)\FD(v)\frac{dv}{\FD(v)}=\rho $} $L^2(V;\frac{1}{\FD(v)})$ then the operator 
	$$L_1f=(\FD(v)\intV{f}-f)$$ 
	can be inverted 
	on the orthogonal complement of $\text{span}(\FD(v))$, the (pseudo)inverse being obtained by multiplication of the right hand side by $-1$, i.e. 
	$$L_1f=\psi \quad \Rightarrow \quad f=-\psi ,\quad \text{for }\psi \in \text{span}(\FD(v))^{\bot }.$$
\noindent
From \eqref{notation-Hs} we deduce
	\begin{align*}
	\left( H_1 - \intV{ H_1 }  \FD(v) \right)  =   \frac{1}{R} v  \FD(v)   
	\end{align*}
	and therefore
	\begin{align*}
	H_1(v)   =    \frac{1}{R}  v  \FD(v).
	\end{align*}
	Moreover,
	\begin{align*}
	L_2 F (v)  =   \hat \lambda_H \nabla_x \VF \ v  \FD(v) 
	\end{align*}
	and
	\begin{align*}
	H_2(v)  = \frac{1}{R}\left (\eta \hat \lambda _H \nabla_x \VF v\FD(v)-v\cdot \nabla_x \FD(v)\right ).
	%\frac{\lambda_H}{\lambda_0} \cdot v \FD(v)	+  \frac{1}{\lambda_0}  v \cdot \nabla_x \FD(v)
	\end{align*}
	\noindent
	This gives
	\begin{align*}
	D  &= \frac{1}{R} \intV{  v \otimes v \FD(v)  } 
	%=: \frac{1}{ \lambda_0 {\cb R}} D_{\FD},%\quad \text{where }D_{\FD}=\intV{  v \otimes v \FD(v)  } 
	\\
	\Gamma  & = \frac{\eta }{ R} \intV { v \otimes v \FD(v)  } \cdot  (\hat \lambda_H \nabla_x \VF) -
	\frac{1} R  \nabla_x \cdot \intV{  v \otimes v \FD(v)  }
	=  \eta D  \hat \lambda_H \nabla_x \VF  - \nabla_x \cdot D
	\end{align*}
	and finally
	\begin{align}
	\partial_t \rho _0  -   \nabla_x \cdot \left( \nabla_x \cdot (\rho _0D)  
	-   \eta \rho _0D \hat \lambda_H  \nabla_x \VF \right)
	=0.
	\label{diffusionequation}
	\end{align}
	\noindent
\noindent
This macroscopic equation involving myopic diffusion and a haptotaxis-like term characterizes the evolution of the tumor as a population of cells guided by the tissue during their migration. It 
corresponds to the one obtained in \cite{EHKS14,EHS}. 
Using the terminology therein we will call $D$ the tumor diffusion tensor. 
Notice that $D$ is also involved in the haptotactic sensitivity coefficient $D\hat \lambda_H$; the latter carries in $\alpha (Q)$ and $g'(Q)$ the information from the subcellular (receptor binding) level and 
by way of $\lambda _0$ and $\eta $ also encodes the cell turning rate, which is a microscopic quantity.

%	together with 
%	\begin{align*}
%	\partial_t Q  = \gamma   S_Q
%	\bar \rho_0 m Q
%	\end{align*}
\end{remark}

\section{First-order moment closures}\label{1st-order-closure}

In the following section we derive first-order moment approximations to the kinetic equation. See \cite{hillenM5,hillen2002hyperbolic} for higher order approximations to related problems. 

\subsection{Balance equations}

We derive equations for the density $\rho = \intV{f}$ and momentum $ q = \intV{ v f }$.
Higher moment approximations can be developed as well, compare for example \cite{Ritter2016}.
We start again with the scaled version of the kinetic equation
\begin{align}
\epsilon^2 \partial_t f  + \epsilon v \cdot \nabla_x f  &=R L_1 f  
+ \epsilon \eta L_2 f.
\label{eq:LKEGeneralScaled}
\end{align}
Integrating with respect to $v$ gives the continuity equation 
\begin{align}
\label{eq-rho-q}
\epsilon \partial_t \rho  +  \nabla_x \cdot q  =0,
\end{align}
which does not depend on the collision operators. 
Equations for the momentum follow from multiplying \eqref{eq:LKEGeneralScaled} with $v$ and integrating over $V$:
\begin{align*}
\epsilon^2 \partial_t q  + \epsilon \nabla_x \cdot P  
= R \intV{v L_1 f}
+ \epsilon \eta \intV{v L_2  f }.
\end{align*}
In the momentum equations, the pressure tensor
$$
P := \intV {v \otimes v  f }
$$
contains the second moments of $\PD$. Since the system is undetermined these equations have to be closed 
by an approximation of $P$ using only $\rho$ and $q$. This is usually obtained by 
defining an ansatz function $f^A (v; \rho,q)$ and using this function to approximate
\begin{align*}
P = \intV{ v \otimes v f } \approx \intV{v \otimes v f^A } =: P^A.
\end{align*}
The closed system of equations is then
\begin{equation}
	\label{eq:firstOrderSystem}
	\begin{aligned}
	\epsilon \partial_t \rho  +  \nabla_x \cdot q  &= 0, \\
	\epsilon^2 \partial_t q  + \epsilon \nabla_x \cdot P^A(\rho,q) &= R \intV{v L_1 f^A(\rho,q)}
	+ \epsilon \eta \intV{v L_2  f^A(\rho,q)}.	
	\end{aligned}
\end{equation}
%$$
% \int_V v^\prime k_1(v^\prime,v)  dv^\prime =0
%$$
%and
%$$
% \int_V  k_1(v^\prime,v)  dv^\prime = \lambda
%$$
\begin{example}
	For the glioma example we have
	\begin{align*}
	\intV{v L_1 f }   &=  - \lambda_0 (q - \rho \intV{ v \FD(v)}), \\
	\intV{v L_2 f }   &= -  (\intV{ v \otimes v f }- q \otimes \intV{v \FD(v)}) \lambda_H
	=  - (P - q \otimes \intV{v \FD(v)}) \lambda_H.
	\end{align*}
	Thus, the momentum equation becomes 
	\begin{align*}
	\epsilon^2 \partial_t q  + \epsilon \nabla_x \cdot P^A  
	=  - \lambda_0 (q - \rho \intV{ v \FD(v)})
	 - \epsilon  (P^A - q \otimes \intV{v \FD(v)})\lambda_H,
	\end{align*}
	which simplifies to 
	\begin{align*}
	\epsilon^2 \partial_t q  + \epsilon \nabla_x \cdot P^A  
	=  - q \lambda_0 
	- \epsilon P^A \lambda_H
	\end{align*}
	for a first-order symmetric fiber distribution $\intV{v\FD(v)} = 0$.
\end{example}

In the following we consider different ansatz functions and show the resulting closure relations for $P^A$. 

\begin{definition}[Normalized moments]
	In the subsequent derivations it will be useful to consider normalized moments indicated by a hat. For example, normalized momentum and pressure tensor are denoted by
	\begin{align*}
	\normal{q} := \frac{q}{\rho}, \quad \normal{P} := \frac{P}{\rho}.
	\end{align*}	
\end{definition}

\begin{remark}
	\label{rem:PinDiffusionLimit}
	In the diffusion limit the distribution can be written as the equilibrium 
	plus $\mathcal{O}(\epsilon)$ perturbations: $f = \rho(t,x) F(v) + \epsilon g$. In order for the moment 
	approximations to converge to the correct limit as $\epsilon \rightarrow 0$, the ansatz should 
	reproduce the correct pressure tensor $P$ when the zeroth- and first-order moments correspond to 
	the equilibrium state; i.e. if $\normal{q} = \intV{Fv}$ then also 
	\begin{align*}
		\intV{ v \otimes v f^A } = \intV{  v \otimes v F(v) } =: D_F
	\end{align*}
	should hold. 
\end{remark}

\subsection{Linear ($P_1^{(F)}$-)closure}
The simple perturbation ansatz
\begin{align*}
f^A= (a + \epsilon v \cdot b) F(v)
\end{align*}
gives the correct pressure tensor $\intV{v\otimes v f^A}$ in the equilibrium $F(v)$. 
The multipliers $a$ and $b$ are chosen to fulfill the moment constraints $\intV{f^A} =\rho$ and $\intV{v f^A} = q $, i.e. they are the solutions of the linear system 
\begin{align*}
  a +  \epsilon b \cdot \intV{ v F} &= \rho,\\
  a \intV{ v F} + \epsilon \intV{ v \otimes v F }   b   &= q,
\end{align*}
or equivalently
\begin{align*}
	a &= \rho - \epsilon b \cdot \intV{ v F}, \\
	\epsilon \left( \intV{v\otimes v F} - \intV{vF}\intV{vF}\trans \right) b &= q - \rho \intV{v F}.
\end{align*}
\begin{remark}
	The above system has a unique solution iff the symmetric matrix $A = \left( \intV{v\otimes v F} - \intV{vF} \intV{v F}\trans \right)$ is invertible. Standard moment theory \cite{Ker76} tells us that the matrix $A$ is positive-semi definite if $F$ is non-negative, or equivalently if the moments $\intV{v\otimes v F}$ and $\intV{vF}$ are second-order realizable. Furthermore, $A$ is strictly positive definite, and therefore invertible, if the moments of $F$ lie in the interior of the realizability domain, which is the case for all $F$ with non-flat support: $supp(F) \nsubseteq E$ for any plane $E$. 
\end{remark}
%We obtain
%\begin{align*}
%  u = \frac{q}{\rho}  =  \frac{\intV{ v F} + \epsilon D_F  b}{1+ \epsilon b \cdot \intV{ v F}}
%\end{align*}
%and
%\begin{align*}
%   \intV{ v F} + \epsilon D_F  b =  u + \epsilon b \cdot \intV{ v F} u
%\end{align*}
%or
%\begin{align*}
%  \epsilon ( D_F   -  u \otimes  \intV{ v F} ) b  =  u -\intV{ v F}.
%\end{align*}
The approximated pressure tensor is 
\begin{align*}  
P^A = \rho \normal{P}^A(\normal{q}),
\end{align*}
with 
\begin{equation}  
	\begin{aligned}
		 \normal{P}^A(\normal{q}) = \frac{\intV{v \otimes v  f^A  }}{\intV { f^A }}=  \frac{\intV{v \otimes v  ( a + \epsilon v \cdot b) F(v)  }}{\intV{ ( a + \epsilon v \cdot b) F(v)  }} \\
		 = \frac{ a D_F  + \epsilon \intV{  v \otimes v v \cdot b F(v)} }{ a+\epsilon \intV{v \cdot b F(v)}},
		 \label{eq:P1P} 
	\end{aligned}
\end{equation}
where $D_F = \intV{v \otimes v F(v)}$ is the pressure tensor of the equilibrium. 
If $F$ is symmetric, i.e. if $\intV{vF(v)} = 0$ and $\intV{v\otimes v vF(v)} = 0$ then the multipliers are simply 
\begin{align*}
	a &= \rho, \\
	b &= (\epsilon D_F)^{-1}q
\end{align*} 
and the pressure tensor becomes $P^A =\rho D_F$. 

\subsection{Nonlinear ($M_1^{(F)}$-)closure}

For this closure we use the approximating function 
\begin{align}
\label{eq:M1Ansatz}
f\sim f^A = a \exp( \epsilon v \cdot b) F(v).
\end{align}
In contrast to the linear closure discussed in the previous section, the ansatz function $f^A$ is now positive, which leads to several advantages
for the resulting approximating equations, see \cite{AniPenSam91,Goudon2005,BruHol01}. The computations proceed in a similar way as before. Again, the multipliers $a $ and $b$  are determined from the moment constraints on $f^A$: 
\begin{align*}
 \rho =  \intV{f^A} &= \intV{a \exp(  \epsilon v \cdot b) F(v) } ,\\
  q = \intV{vf^A} &= \intV{ v a \exp(  \epsilon  \cdot b) F(v) } .
\end{align*}
This gives 
\begin{align*} 
  \normal{q}(b) = \frac{\intV{ v  \exp(  \epsilon v \cdot b) F(v) }} {\intV{  \exp(  \epsilon v \cdot b) F(v) }}
\end{align*}
and
\begin{align*} 
P^A = \rho \normal{P}^A(\normal{q}),
\end{align*}
with
\begin{align}
\label{eq:M1SecondMoment}
\normal{P}^A(\normal{q}) = \frac{\intV{ v \otimes v   \exp( \epsilon v \cdot b) F(v) }} {\intV{ \exp(\epsilon  v \cdot b) F(v) }}.
 \end{align}
 %\begin{align*}
 %\int_V f(v) k_1 [m,Q](v,y)  dv = \int_V (a \exp(v \cdot b )) F(v) k_1[m,Q](v,y)  dv\\
 %= \rho \frac{\int_V ( \exp(v \cdot b )) F(v) k_1[m,Q](v,y)  dv}{\int_V  \exp( v \cdot b) F(v) dv}
 %= \rho T_E (b)
 %\end{align*} 

\subsection{Simplified nonlinear closure ($K_1^{(F)}$)}

We assume  $F\geq 0$, $\intV{F} =1$, and $\intV{vF} =0$. Remember that this implies that 
$\tr(<v\otimes vF>) = \tr(D_F) = 1$.
Now we want to extend the concept of Kershaw closures  \cite{Ker76} for our special situation. We determine the second moment $P^A$ via an interpolation between the free-streaming value  $P_{\delta} = \rho \frac{q\otimes q}{\vert q\vert^2}$ for $\vert q \vert = \rho$ and the equilibrium solution $P_{eq} = \rho D_F$ for $\vert q \vert=0$ (compare \eqref{eq:M1SecondMoment} with $b = 0$) and make the ansatz
\begin{align}
\label{eq:KershawP}
P^A = \rho  \normal{P}^A(\normal{q}) := \rho\left(\alpha D_F + (1-\alpha)\frac{\normal{q}\otimes\normal{q}}{\vert \normal{q}\vert^2}\right),
\end{align}
where $\alpha= \alpha(\normal{q})$ is given below.
To obtain a reasonable model it is crucial to satisfy the so-called realizability conditions \cite{Ker76,Schneider2014,Schneider2015c}, i.e. the fact that the moments $\normal{q},\normal{P^A}$ are generated by a non-negative distribution function. In this case we have to ensure that for every $\rho \geq 0$ and $\vert \normal{q} \vert \leq 1$ we have that \cite{Ker76}
\begin{align*}
\normal{P}-\normal{q}\otimes \normal{q} \geq 0 \; \; \mbox{and} \;\; \tr(\normal{P}) = 1.
\end{align*}
The trace equality immediately follows for all $\alpha\in\R$ since $\tr(D_F) = \tr(\frac{\normal{q}\otimes\normal{q}}{\vert \normal{q}\vert^2}) = 1$.
Plugging in the definition of $P^A$ gives that 
\begin{align*}
\normal{P}^A-\normal{q}\otimes \normal{q} = \alpha D_F + (1-\alpha - \vert \normal{q}\vert^2)\normal{q}\otimes \normal{q}
\end{align*}
is positive semidefinite if $\alpha\geq 0 $ and $1-\alpha \geq \vert \normal{q}\vert^2$.
We use 
\begin{align}
	\alpha = 1-\vert \normal{q}\vert^2,
	\label{kershawInterpolationParam1}
\end{align} 
which satisfies both inequalities under the first-order realizability condition $\vert \normal{q}\vert \leq 1$.
Note that in the special case $D_F = \frac{I}{3}$ the original Kershaw model \cite{Ker76} is recovered.

\begin{thm}[Hyperbolicity of the generalized symmetric Kershaw moment system]\label{only-theorem}
	For any distribution $\FD : \US \mapsto \R^+$ that is
	\begin{itemize}
		\item normalized: $\intV{\FD} = 0$, 
		\item symmetric w.r.t to the first moment $\intV{\FD v} = 0$,
		\item non-flat: $\intV{\FD (x\trans v)^2} > 0, \forall x \in \US$
	\end{itemize}
	the first order moment system \eqref{eq:firstOrderSystem} together with the Kershaw closure \eqref{eq:KershawP} is strictly hyperbolic for all realizable moment vectors $(\rho,q)$, except for $|\normal{q}| = 1$ with $\normal{q}$ parallel to an eigenvector of $\intV{\FD vv\trans}$. In this case the system matrix still has real eigenvalues but cannot be diagonalized any more.   
	\label{thm:KershawHyperbolicity}
\end{thm}
\begin{proof}
	See \ref{app:HyperbolicityKershaw}.
\end{proof}

\section{Higher-order moment models}\label{higer-order-closure}
Analogously to the first-order moment system, higher-order approximations are conceivable. Let $\basisf_N(v) = (a_0(v), ... a_{K-1}(v))$ be the basis of a $K$- dimensional subspace of $L_2(V)$ that 
contains polynomials of up to order $N$. 
The corresponding moments are defined as $\momf _N:= \intV{\PD \basisf_N}$. By multiplying the kinetic equation  \eqref{pdegeneraleps} with $\basisf_N$ and integrating over $V$ we get a system for 
the moments:
\begin{align}
\label{eq:MomentSystem}
\partial_t \momf_N + \frac{1}{\epsilon}\nabla_x \cdot \intV{v \basisf_N \PD}  %=\intV{\frac{St^2}{Kn} \frac{1}{\epsilon^2} L_1(\PD)  + \frac{\eta}{\epsilon}  L_2(\PD) \basisf_N}.
= \intV{(\frac{R}{\epsilon ^2}L_1(\PD)+\frac{\eta}{\epsilon}L_2(\PD))\basisf_N}.
\end{align}
As in the first-order case, $\PD$ is approximated by an ansatz function
\begin{align*}
\PD^A[\momf_N](v) \approx \PD(v),
\end{align*} depending on the moments, to get a closed form
\begin{align*}
\partial_t \momf_N + \nabla_x \cdot \intV{v \basisf_N \PD^A} = \frac{R}{\epsilon ^2} \intV{L_1(\PD^A) \basisf_N} + \frac{\eta}{\epsilon} \intV{L_2(\PD^A))\basisf_N}
\end{align*}
of the moment system. 
One choice for the basis $\basisf_N$ are monomial functions 
\begin{align*}
a_k = v^{\pmb{i}(k)} = v_x^{i_x(k)} v_y^{i_y(k)} v_z^{i_z(k)},
\end{align*}
where $\pmb{i}(k)$ is a bijective mapping from basis indices $k = 0,1,...,K(N)-1$ to multi-indices $\pmb{i} \in \pmb{I}(K)$. 

\noindent
The classical $P_N$ and $M_N$ methods use the ansatz functions
\begin{align*}
\PD^A = \mplf_N \cdot \basisf_N \qquad \text{ and } \qquad 
\PD^A = \exp(\mplf_N \cdot \basisf_N),
\end{align*} 
respectively. Analogously to the first-order methods we define the modified $P^{(F)}_N$ and $M^{(F)}_N$ as 
\begin{align*}
\PD^A = (\mplf_N \cdot \basisf_N) F(v)\qquad \text{ and } \qquad 
\PD^A = \exp(\mplf_N \cdot \basisf_N) F(v),
\end{align*}
respectively, in order to incorporate the equilibrium of the reorientation kernel $F(v)$.

%%%%%%%%%%%%%%%%%%%%%%%%%%%%%%%%%%%%%%%%%%%%%%%%%%%%%
\section{Numerical results}
\label{sec:Results}
As already mentioned before, minimum-entropy and Kershaw closures are only well-defined for realizable moment vectors.  It is easy to show that a standard first-order scheme provides this property under a CFL condition that heavily depends on the chosen physical parameters \cite{Schneider2015a,Schneider2016,Schneider2016aCodeIMEX}. 
To increase the efficiency of our approximation, we use the second-order realizability-preserving scheme presented in ~\ref{app:SecondOrderIMEX}. Due to space limitations we postpone the analysis of the scheme indeed pre\-ser\-ving realizability to a follow-up paper.

\subsection{Numerical experiments}
All numerical simulations will be done for the glioma equation \eqref{kineticsystemreduced}. We have not yet specified the form of the equilibrium $\FD(v)$. 
We use a quadratic ansatz for the fiber distribution
\begin{align*}
	 \FD(v) &=  \frac{3}{4 \pi \trace{D_W}} \left(v\trans D_W v \right).
\end{align*} This so-called peanut distribution \cite{EHKS14} is a very simple model that relates the fiber distribution to the local water diffusion tensor $D_W \in \R^{3 \times 3}$,  which can be measured by DTI \cite{LeBihan2001DTI}. It has the additional advantage that all the coefficients in the diffusion limit can be computed analytically.  

We estimate the volume fraction $\VF$ from  the local water diffusion tensor $D_W$,  using either the fractional anisotropy or the characteristic length. The fractional anisotropy $FA(D_W)$ estimate from \cite{EHKS14} is 
\begin{align*}
\VF(x) &= FA(D_W(x)) := \sqrt{\frac{3}{2} \frac{\sum_{i=1}^3 (\lambda_i - \bar{\lambda})^2}{\sum_{i=1}^3 \lambda_i^2}},
\end{align*}
with the eigenvalues $\lambda_i$ of the water diffusion tensor $D_W$ and the mean eigenvalue $\bar{\lambda} =\frac{1}{3} \sum_{i=1}^3 \lambda_i$. The characteristic length estimate $CL(D_W)$ from \cite{EHS} leads to
\begin{align}\label{estimated-Q}
    \VF(x) &:= CL(D_W(x)) = 1 - \left( \frac{\trace{D_W}}{4 \lambda_1} \right)^{\frac{3}{2}},
\end{align}
where $\lambda_1$ is the maximum eigenvalue of $D_W$. 
%The turning rate $\lambda_H$ is chosen according to \eqref{turningrates}.
In the moment models we use mass conserving, thermal boundary conditions for the incoming characteristics
\begin{align*}
(v \cdot n) \PD &= (v \cdot n) \FD(v) \int_{\{v' \cdot n > 0\}} |v' \cdot n| \PD(v') dv', & v \cdot n < 0,
\end{align*}
with the unit outer normal $n$. This means that outgoing particles are absorbed at the boundary and emitted according to the fiber distribution. For the diffusion approximation the only condition is that there is no flux over the boundary.

Finally, to compare different models we define the pointwise relative difference between two functions $h_1(x), h_2(x)$ as 
\begin{align*}
	\relerr{h_1(x)}{h_2(x)} &= \frac{\left| h_1(x) - h_2(x) \right|}{\norm{h_2(x)}_{\infty}}.
\end{align*}

\subsubsection{Abruptly ending fiber strand}
This setting models an initially concentrated mass of cells following a white matter tract that abruptly ends. While the latter is not to be expected for a real brain geometry, we use it in order to show some notable effects in the glioma equation. The involved diffusion tensor and volume fraction are both spatially varying.
The computational domain is
\begin{align*}
	\Omega_{TXV} &= [0,T] \times [0,X]^2 \times c\,\US ,\\
	T &= 2,~ 
	X = 3,~
	c = \frac{X}{ \epsilon T},
\end{align*}
where the cell speed  $c$ is chosen to adjust the parabolic scaling parameter $\epsilon = St = \frac{X}{cT}$ from \eqref{pdegeneraleps}. The coefficients $\lambda_0, \lambda_1$ involved in the turning rate are chosen such that $R = \eta = 1$:
\begin{align*}
	\lambda_0 &= c^2 \frac{T}{X^2} = \frac{1}{\epsilon^2 T},\\
	\lambda_1 &= \lambda_0.
\end{align*}
From \eqref{turningrates} we see that the attachment and detachment rates $k^+ , k^{-} $ have to be scaled together with the turning rate coefficient $\lambda_0$ and hence set them all equal $k^+ = k^{-} = \lambda_0$.
The fiber geometry is modeled by setting the water diffusion tensor $D_W(x)$ as a function of space. We use a diagonal matrix
\begin{align*}
	D_W(x) &= 
	\begin{pmatrix}
	D_{00}(x) & 0 \\ 0 & 1
	\end{pmatrix}, \\
	D_{00}(x) &= 1 + 5 \exp\left(-\frac{\nu(x)}{2 \sigma^2}\right),\\
	\nu(x) &= \max\left\lbrace 0, x_1 - \frac{X}{2}, |x_2 - \frac{X}{2}| - 0.1 \right\rbrace,
\end{align*}
where only the first eigenvalue varies in space to blend smoothly between a strongly concentrated distribution in x-direction with $D_{00} = 6$ and an isotropic distribution $D_{00} = 1$. For simplicity, the volume fraction $\VF(x)$ is then computed as $FA(D_W(x))$.

The initial condition is a square of length $0.1$ centered at $(0.5,1.5)$:
\begin{align*}
	f(t=0,x,v) = \frac{1}{4\pi} 
	\begin{cases}
		1 & x \in [0.45,0.55]\times[1.45,1.55],\\
		10^{-4} & \mbox{else}.
	\end{cases}
\end{align*}
\paragraph{Convergence to the diffusion limit}
We compare the diffusion approximation from \eqref{diffusionequation} to various moment models as $\epsilon \rightarrow 0$. 
\figref{fig:diffusionlimit} shows the solution for the diffusion approximation $D$(\ref{fig:diffusionlimitD}) alongside the first-order Kershaw method $K_1^{(F)}$(\ref{fig:diffusionlimitK1} - \ref{fig:diffusionlimitK001}) at $\epsilon = 1,0.5,0.25,0.1,0.01$. Additionally, the pointwise relative difference between both models (\ref{fig:diffusionlimitKD025} - \ref{fig:diffusionlimitKD001}) at $\epsilon = 0.25,0.1,0.01$ is shown. Far from the diffusion limit at $\epsilon = 1$ the cells travel exactly once through the domain ($St = 1$) and have an expectation of one ($Kn = 1$) velocity jump. Therefore, we see a strongly advection dominated behavior with very little influence from the underlying fiber distribution. As $\epsilon$ gets smaller, 
the Kershaw model becomes increasingly similar to the diffusion approximation. Relative pointwise differences also decrease although even at $\epsilon =  0.01$ some discrepancy in the range of $2\% - 5\%$ remains. We attribute this to inherent differences between the numerical schemes at a not yet fine enough grid. 
As a re\-fe\-rence we show the standard $P_5$ solution in \figref{fig:P5} at $\epsilon = 0.1$ and $\epsilon =0.01$.  Note that higher moment order $P_N$ models are not shown here as they do not differ significantly from the $P_5$ solution. From the relative difference to $P_5$ it is apparent that at $\epsilon = 0.01$ the solution is so close to the diffusion limit that moment models yield only a marginal improvement.  Again, the remaining $2\%$ difference in \ref{fig:P5diffD001} can be attributed to the differences in the numerical schemes. However, at $\epsilon=0.1$ the diffusion approximation starts to lose validity and deviates from the $P_5$ solution over $10\%$ in places. The first-order Kershaw model gives a noticeable improvement in this case although a difference of $5\%$ remains. 

\paragraph{Modified and standard moment models}
\figref{fig:ComparePN}~ shows the solution at $\epsilon = 0.1$, both for the standard $P_1,P_3,P_5$ models and the modified $P_1^{(F)},P_3^{(F)},P_5^{(F)}$ models. The standard $P_1$ model cannot represent the correct pressure tensor if the distribution is in equilibrium $\PD(x) = \FD(x)$ and thus does not converge to the diffusion limit. This can be observed in \ref{fig:ComparePN_P1}. Here, the modified model $P_1^{(F)}$ that includes $\FD$ in the ansatz function leads to a great improvement. For the higher moment-orders the difference between standard and modified models becomes less pronounced. This is to be expected since in the special case of $\FD = v\trans D v$, the $P_3$ model already contains $\FD$. 

\begin{figure}
	\def\localpath{Images/Fibreend/}
	\centering
	\externaltikz{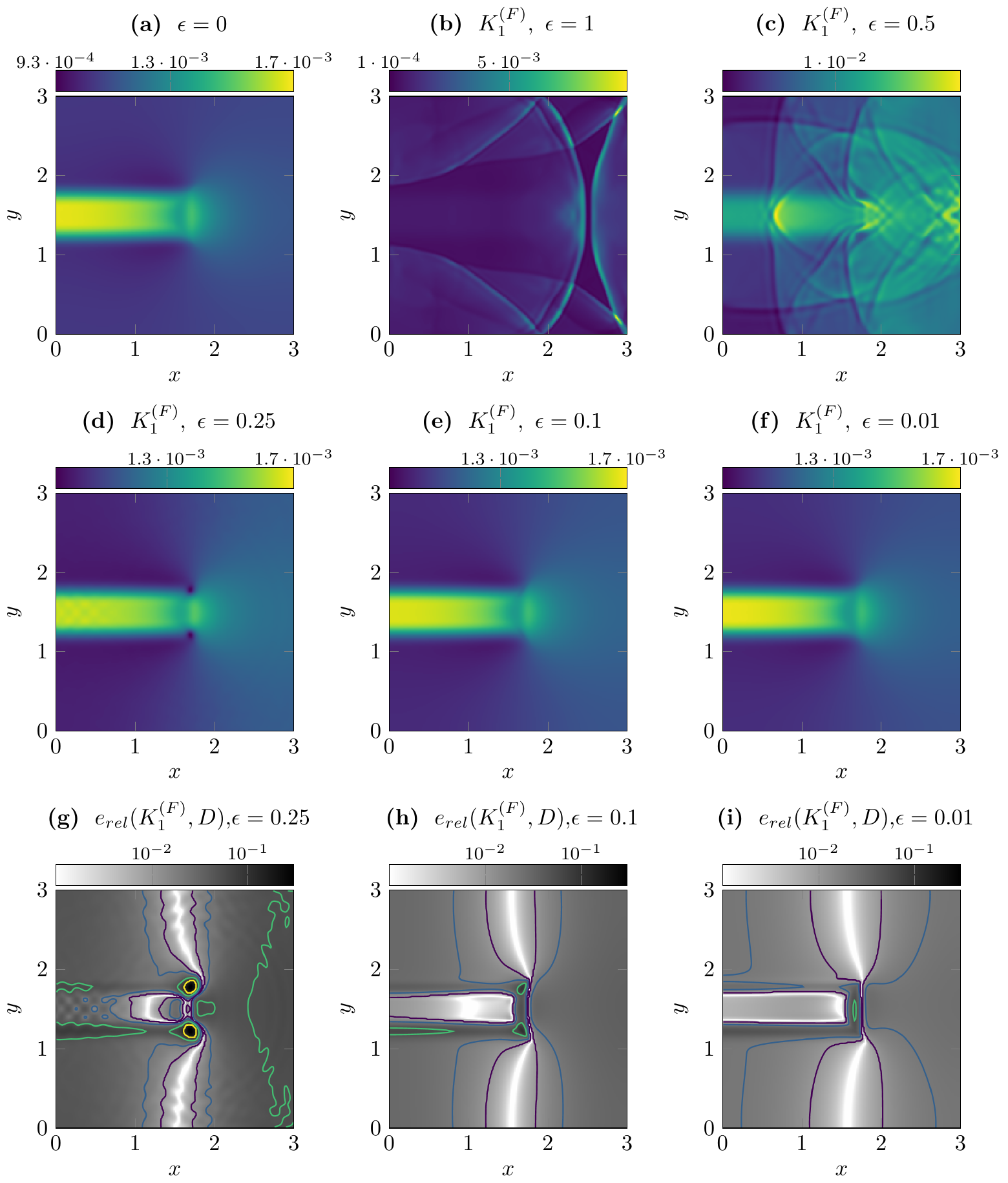}{\input{Images/Fibreend/DiffusionLimit}}
	\settikzlabel{fig:diffusionlimitD}\label{diff-approx}
	\settikzlabel{fig:diffusionlimitK1}\label{K1-eps1}
	\settikzlabel{fig:diffusionlimitK05}\label{K1-eps05}
	\settikzlabel{fig:diffusionlimitK025}\label{K1-eps025}
	\settikzlabel{fig:diffusionlimitK01}\label{K1-eps01}
	\settikzlabel{fig:diffusionlimitK001}\label{K1-eps001}
	\settikzlabel{fig:diffusionlimitKD025}\label{KD-eps025}
	\settikzlabel{fig:diffusionlimitKD01}\label{KD-eps01}
	\settikzlabel{fig:diffusionlimitKD001}\label{KD-eps001}
	\caption{Comparison between the $K_1^{(F)}$-model and the diffusion approximation as $\epsilon$ approaches $0$. (\ref{diff-approx}): the diffusion approximation. (\ref{K1-eps1})-(\ref{K1-eps001}): $K_1^{(F)}$ at $\epsilon = 1,0.5,0.25,0.1,0.01$, respectively. (\ref{KD-eps025})-(\ref{KD-eps001}): the relative difference between the diffusion and the Kershaw model $\relerr{K_1}{D}$ for $\epsilon = 0.25,0.1,0.01$, respectively. Differences are plotted on a logarithmic scale. Contours are drawn for $0.1$ (yellow), $0.05$ (green), $0.02$ (blue), $0.01$ (purple).}
	\label{fig:diffusionlimit}
\end{figure}

\begin{figure}
	\def\localpath{Images/Fibreend/}
	\externaltikz{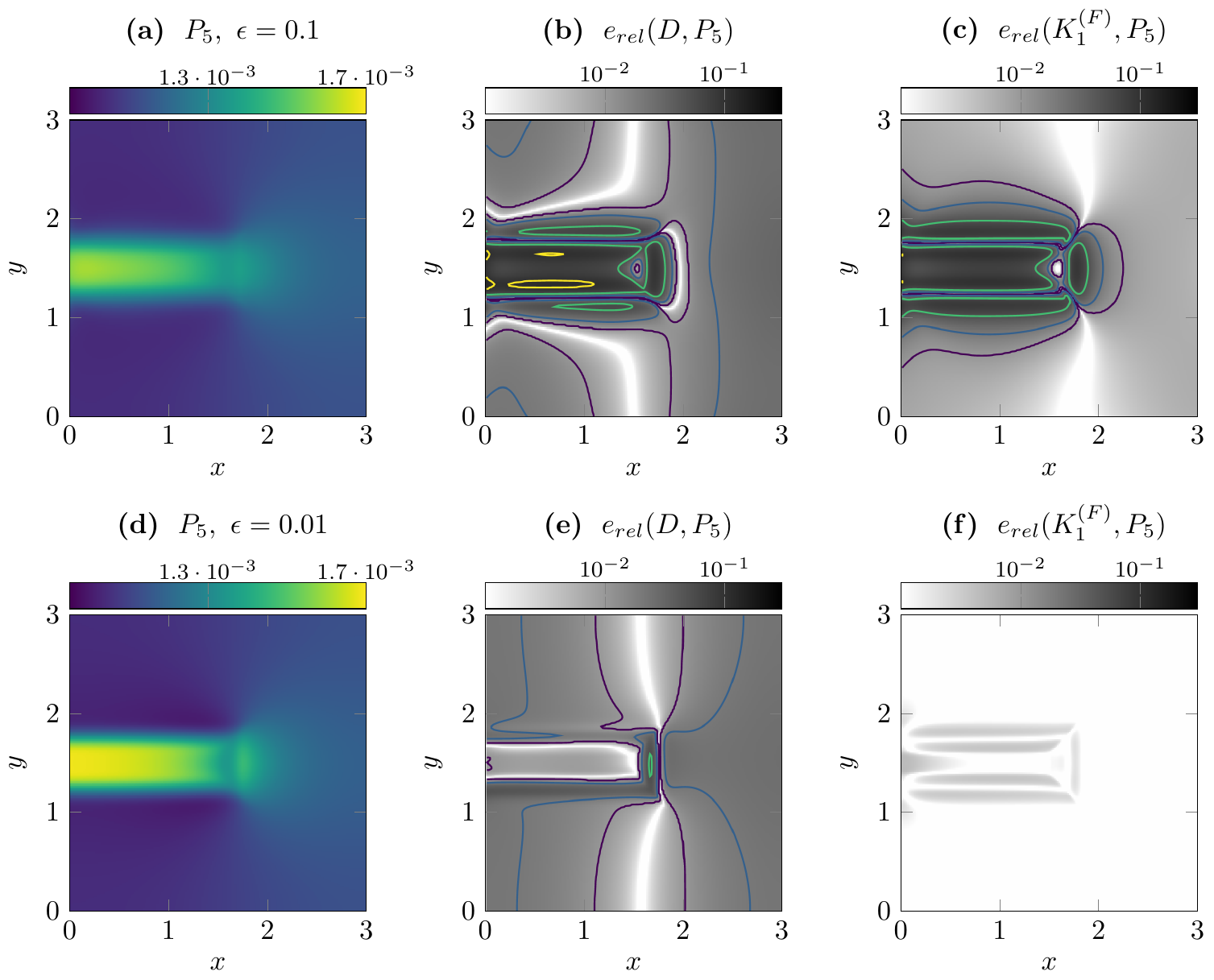}{\input{Images/Fibreend/P5}}
	\settikzlabel{fig:P5solution01}\label{P5-eps01}
	\settikzlabel{fig:P5diffD01}\label{DP5-eps01}
	\settikzlabel{fig:P5diffK01}\label{K1P5-eps01}
	\settikzlabel{fig:P5solution001}\label{P5-eps001}
	\settikzlabel{fig:P5diffD001}\label{DP5-eps001}
	\settikzlabel{fig:P5diffK001}\label{K1P5-eps001}
	\caption{The $P_5$ solution 
	%at $\epsilon = 0.1$ (subfigure (\ref{P5-eps01})) and at $\epsilon = 0.1$ (subfigure (\ref{P5-eps001}))
	and relative difference to $K_1^{(F)}$ and diffusion approximation. Upper row: $\epsilon =0.1$, lower row: $\epsilon=0.01$. Differences are plotted on a logarithmic scale. Contours are drawn for $0.1$ (yellow), $0.05$ (green), $0.02$ (blue), $0.01$ (purple).}
	\label{fig:P5}
\end{figure}

\begin{figure}
	\def\localpath{Images/Fibreend/}
	\externaltikz{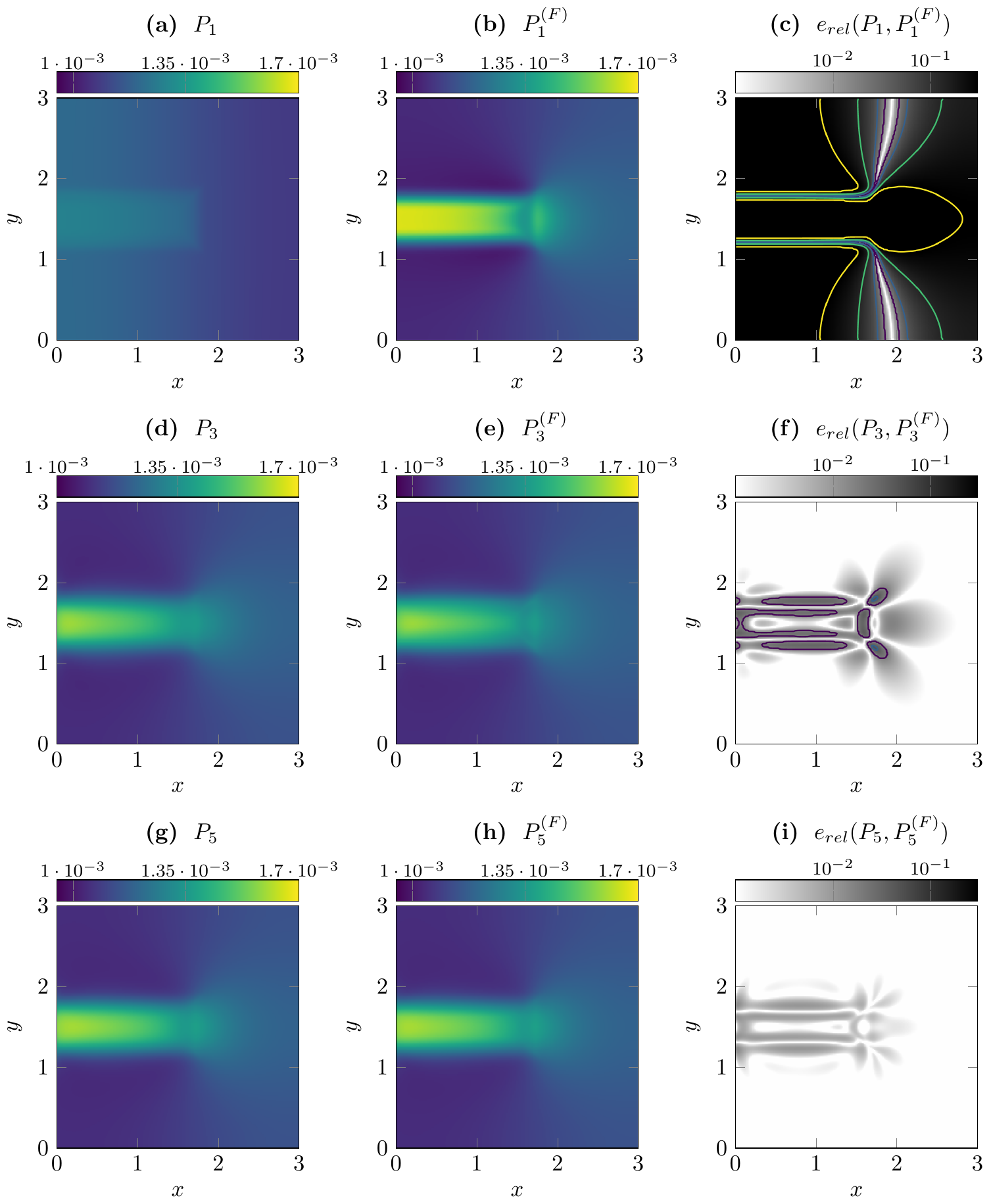}{\input{Images/Fibreend/ComparePN}}
	\settikzlabel{fig:ComparePN_P1}
	\settikzlabel{fig:ComparePN_P1F}
	\settikzlabel{fig:ComparePN_P1diff}
	\settikzlabel{fig:ComparePN_P3}
	\settikzlabel{fig:ComparePN_P3F}
	\settikzlabel{fig:ComparePN_P3diff}
	\settikzlabel{fig:ComparePN_P5}
	\settikzlabel{fig:ComparePN_P5F}
	\settikzlabel{fig:ComparePN_P5diff}
	\caption{Comparison between standard and modified $P_N$ models for $N=1$ (upper row), $N=3$ (middle), and $N=5$ (lower row). Color code of contours is the same as in Figures \ref{fig:diffusionlimit} and \ref{fig:P5}.}
	\label{fig:ComparePN}
\end{figure}

\subsubsection{2D Brain slice}
In this numerical experiment we take water diffusion tensors $D_W$ from a DTI scan of the human brain\footnote{Provided by Carsten Wolters (Institute for Biomagnetism and Biosignal Analysis, WWU M\"unster).}. The tensor field is visualized in \figref{fig:Brain_fibres}.  We use the characteristic length 
estimate $CL(D_W)$ to obtain the volume fractions via \eqref{estimated-Q}, which are shown in \figref{fig:BrainFA}. Additionally the main diffusion direction, i.e the largest eigenvector of $D_W$, is shown in \figref{fig:BrainRGB} as a four-channel color-coded image.  
The initial tumor mass, marked by the white square in \figref{fig:BrainFA}, is concentrated in a square of length $5mm$
\begin{align*}
f(t=0,x,v) = \frac{1}{4\pi} 
\begin{cases}
1 & x \in [98.5, 103.5]\times[158.5,163.5]\\
10^{-4} & \mbox{else}
\end{cases}
\end{align*}
at the center of the spatial domain $\Omega_X = [50, 150] \times [110, 210]$, which is indicated by the red square in \figref{fig:BrainFA}. All other physical parameters are listed in \tabref{tab:parameters}. Note that the values we use here are quite far from the actual measured parameters used in \cite{EHKS14,EHS}. Therein the characteristic numbers are $St \approx 0.03$ and $Kn \approx 8 \times 10^{-8}$ and so $R \approx 10^{4}, \eta \approx 1000$, leading to little diffusion and a very pronounced migration along fibers. Due to a stiff right-hand-side and a restrictive $CFL$-condition the numerical method for the moment models introduces too much artificial diffusion in this regime to provide meaningful results. 
With the current set of parameters the characteristic behavior of glioma cells can still be observed, i.e predominant movement and concentration along white matter tracts. At $\epsilon \approx 0.3$, while the diffusion approximation and moment models are structurally similar, there is a significant difference of up to $20\%$ between them. However, a first-order moment approximation seems to be accurate enough since the difference to the third-order $P_3^{(F)}$ is mostly below $2\%$.

{
	\renewcommand{\arraystretch}{1.2}
	\begin{table}
		\begin{tabular}{l@{\hspace{5pt}}lr@{.\hspace{0pt}}l@{$\times$\hspace{0pt}}lcl}
			\multicolumn{2}{l}{Parameter}   &\multicolumn{3}{l}{Value}&                  & Description \\
			\hline 
			T           && $1$&$5768$&  $10^{ 7}$  & $s$              & time span = half a year \\
			c           && $2$&$1$   &  $10^{-4}$  & $ \frac{mm}{s}$  & cell speed \\
			$\lambda_0$ && $1$&$0$   &  $10^{-5}$  & $\frac{1}{s}$    & cell-state independent part of turning rate\\
			$\lambda_1$ && $2$&$5$   &  $10^{-4}$  & $\frac{1}{s}$    & cell-state dependent part of turning rate\\
			$k^+$       && $1$&$0$   &  $10^{-5}$  & $\frac{1}{s} $   & attachment rate of cells to ECM\\
			$k^-$       && $1$&$0$   &  $10^{-5}$  & $\frac{1}{s} $   & detachment rate of cells to ECM\\	
			%Number      &\multicolumn{3}{l}{Value}&                  & Interpretation \\
			\hline		
			$\epsilon$&$=St$&$3$&$02$  &  $10^{-1}$  &                  & Strouhal number\\
			$Kn         $&&$6$&$34$  &  $10^{-3}$  &                  & Knudsen number\\
	$R$&$= \frac{St^2}{Kn}$&$1$&$44$  &  $10^{1}$  &                  & \\
	        $\eta$&$= \frac{\lambda_1}{\lambda_0}$&$2$&$5$& $10^{1}$  && Ratio of turning rate coefficients\\ 
		\end{tabular}
		\caption{The parameters and the resulting characteristic numbers used in the 2D brain slice simulation.}
			\label{tab:parameters}
	\end{table}
}

\begin{figure}
	\def\localpath{Images/Brain/}
	\settikzlabel{fig:BrainFA}\label{Q-CL}
	\settikzlabel{fig:BrainRGB}\label{RGB}
	\withfiguresize{1.8205\figurewidth}{1.5\figurewidth}{\externaltikz{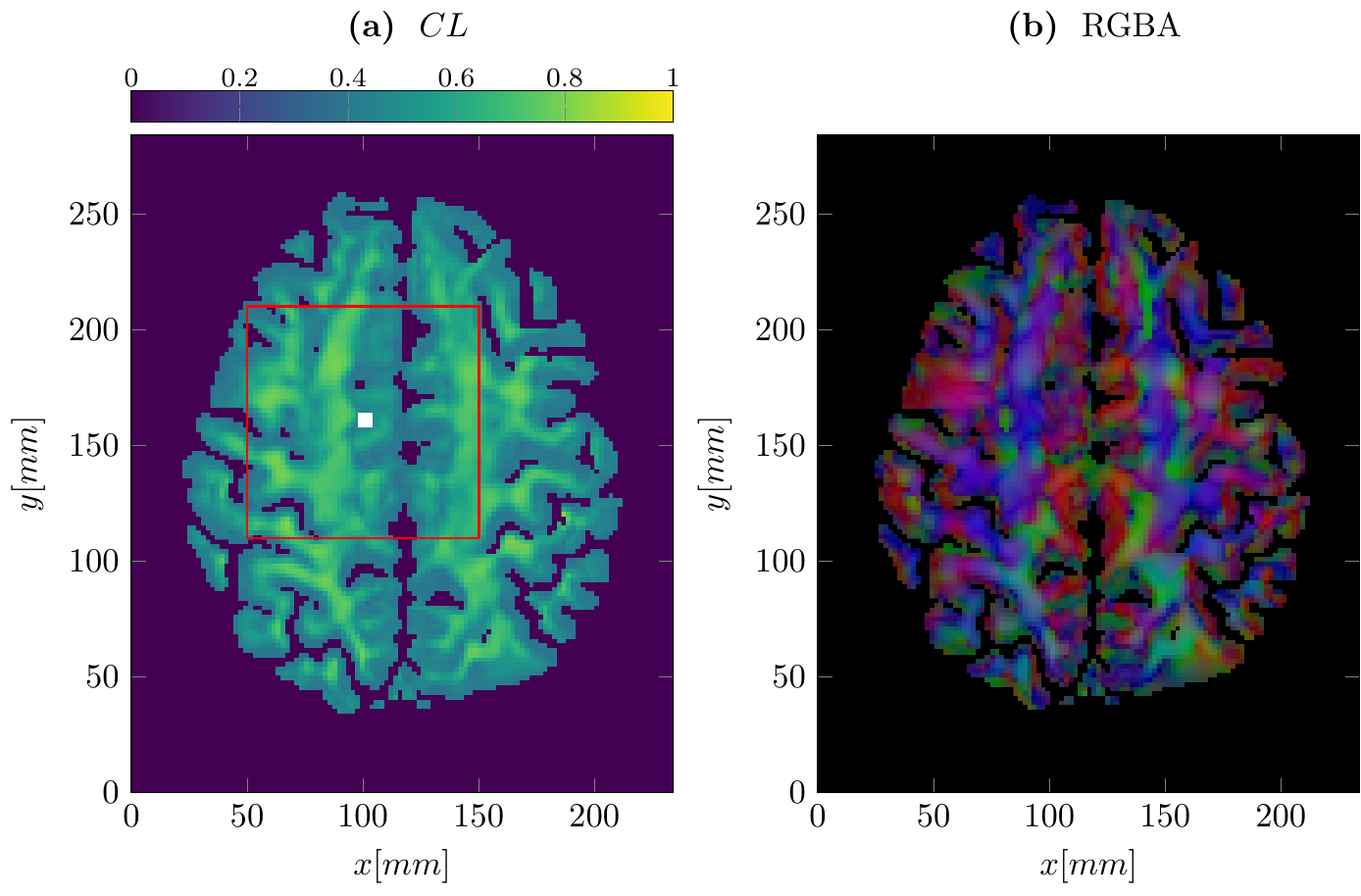}{\input{Images/Brain/Brainfibres}}}
	\caption{Estimated volume fraction $\VF=CL$ via characteristic length (subfigure (\ref{Q-CL})) and an RGBA coded image of the main axis of $D_W$ (sufigure (\ref{RGB})). The color channels RGB encode the $x,y,z$ components of the eigenvector corresponding to the largest eigenvalue of $D_W$ respectively, while the alpha channel is scaled with $CL$. The red square indicates the computational domain, while the white square marks the initial cell distribution.}
	\label{fig:Brain_fibres}
\end{figure}

\begin{figure}
	\def\localpath{Images/Brain/}
	\settikzlabel{fig:BrainPartCL_D}\label{CL-D}
	\settikzlabel{fig:BrainPartCL_P1}\label{CL-P1}
		\settikzlabel{fig:BrainPartCL_K1}\label{CL-K1}
	\settikzlabel{fig:BrainPartCL_P3}\label{CL-P3}
		\settikzlabel{fig:BrainPartCL_D-P3}\label{CL-D-P3}
	\settikzlabel{fig:BrainPartCL_P1-P3}\label{CL-P1-P3}
	\withfiguresize{\figurewidth}{\figurewidth}{\externaltikz{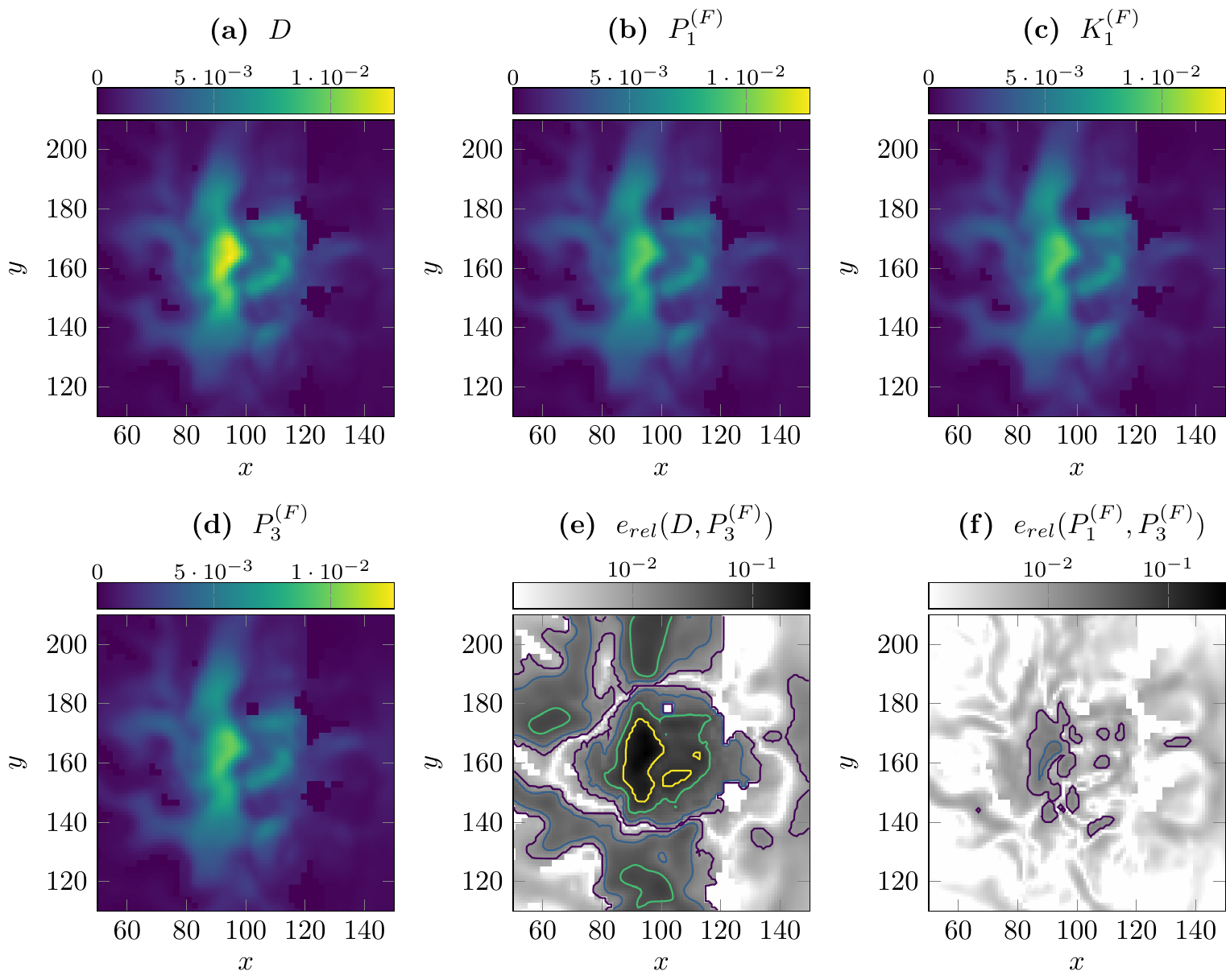}{\input{Images/Brain/BrainPartCL}}}
	\caption{Cell distribution at end time $t = T$. (\ref{CL-D}): the diffusion approximation $D$, and 
	(\ref{CL-P1})-(\ref{CL-P3}): moment models $P_1^{(F)}$,$K_1^{(F)}$,$P_3^{(F)}$, respectively. (\ref{CL-D-P3})-(\ref{CL-P1-P3}): relative differences. Color code of contours is the same as in Figures \ref{fig:diffusionlimit} and \ref{fig:P5}.}
	\label{fig:BrainPartCL}
\end{figure}

%%%%%%%%%%%%%%%%%%%%%%%%%%%%%%%%%%%%%%%%%%%%%%%%%%%%%
\section{Concluding remarks}
\label{sec:Remarks}
We investigated the use of first and higher-order moment closures in comparison to the diffusion approximation in the context of glioma migration in the human brain, characterized by a macroscopic setting involving a haptotaxis-like term. The latter was obtained due to considering a supplementary kinetic 'activity' variable modeling the subcellular level dynamics of receptor binding to the surrounding tissue. In the considered examples the moment models converge as expected to the correct diffusion limit, while being more accurate in situations further from the diffusion limit. Using the modified models that explicitly contain the equilibrium distribution in their ansatz functions leads to improved results for lower moment-orders, while the difference between models vanishes as the moment-order gets larger. 
There is a range of parameters where the diffusion approximation is not accurate enough but a low-order moment model would be sufficient. This could be of relevance for the estimation of the actual tumor extent for the envisaged application. However, we could not compare the methods for parameters that are meaningful in glioma invasion as the numerical scheme for the moment models produces too much artificial diffusion in this regime. Further work will be devoted to develop an asymptotic-preserving scheme. 

\subsection*{Acknowledgment}  
We would like to thank Christian Engwer from WWU M\"unster for patiently answering all questions concerning the implementation in DUNE. 
\appendix

\section{Second-order realizability-preserving scheme}
\label{app:SecondOrderIMEX}

Throughout this Appendix we mean by 'cell' the numerical discretization element and not the biological 
entity.

\subsection{The scheme}
We consider the general one-dimensional hyperbolic system 
\begin{align}
\label{eq:GeneralHyperbolicSystem1D}
\dt\moments+\dx\Flux\left(\moments\right) = \Source\left(\moments\right),
\end{align}
which is in our case the moment system \eqref{eq:MomentSystem}. A second-order realizability-preserving scheme can be obtained by using the operator splitting approach
\begin{alignat}{2}
	\label{eq:SplittedSystema}
	\dt\moments&+\dx\Flux\left(\moments\right) &&= 0, \\
	\label{eq:SplittedSystemb}
	\dt\moments& &&= \Source\left(\moments\right),
\end{alignat}
with Strang splitting and realizability-preserving second-order schemes for the two subproblems. Note that a generalization to a regular grid in two or three dimensions is straightforward.

\subsubsection{Flux system}

Let us first consider the non-stiff part \eqref{eq:SplittedSystema}. This can be solved using standard methods, which will be summarized in the following.

The standard finite-volume scheme in semidiscrete form for \eqref{eq:SplittedSystema} looks like
\begin{align}
\label{eq:FV1}
\dt \momentscellmean{\cellind} = \numericalFlux(\momentspv{\cellind+\frac12}^+,\momentspv{\cellind-\frac12}^-)-\numericalFlux(\momentspv{\cellind+\frac12}^+,\momentspv{\cellind-\frac12}^-),
\end{align}
where $\momentscellmean{\cellind}$ is the average of $\moments$ over cell $\cellind$ and $\numericalFlux$ is a numerical flux function. We use the global Lax-Friedrichs flux
\begin{align}
\label{eq:globalLF}
\numericalFlux(\moments[1], \moments[2]) = \dfrac{1}{2} \left( \Flux_3(\moments[1]) + \Flux_3(\moments[2]) - \viscosityconstant ( \moments[2] - 
\moments[1]) \right).
\end{align}
The numerical viscosity constant $\viscosityconstant$ is taken as the global estimate of the 
absolute value of the largest eigenvalue of the Jacobian
$\Flux'$. In our case, the viscosity constant can be set to $\viscosityconstant = \frac{1}{\epsilon}$ because for the moment systems used here it is known that the largest eigenvalue is bounded in absolute value by the cell speed \cite{Schneider2015a,Schneider2016}.

The value $\momentspv{\cellind+\frac12}$ is the evaluation of a limited linear reconstruction of $\moments$ at the cell interface
\begin{align*}
\momentslocal{\cellind}(x) &= \momentscellmean{\cellind}+\momentslocal{\cellind}'\left(x-x_\cellind\right),\\
\momentslocal{\cellind}' &= \frac{1}{\dxstepsize}S\left(\momentscellmean{\cellind+1}-\momentscellmean{\cellind},\momentscellmean{\cellind}-\momentscellmean{\cellind-1})\right),
\end{align*}
applied component-wise.
As the scheme is of order two overall, it is sufficient to use a second-order central WENO\cite{Shu2003,Dumbser2008} function $S(\cdot) = W_2(\cdot)$ 
\begin{align*}
W_2(a_1,a_2) &= \frac{w(a_1)a_1 + w(a_2)a_2}{w(a_1) + w(a_2)},\\
w(a) &= (\theta + \dxstepsize a)^{-z}
\end{align*} as limiter. 
We set $\theta = 10^{-6}, z = 2$. To avoid spurious oscillations, the reconstruction has to be performed in characteristic variables \cite{Toro2009,Olbrant2012,Chidyagwai2017}. To ensure the realizability-preserving property, we additionally use the realizability limiter derived in \cite{Olbrant2012,Chidyagwai2017,Schneider2016}. If we discretize \eqref{eq:FV1} with a second-order SSP scheme, e.g. Heun's method or the general $\RKstages$ stage SSP ERK${}_2$ \cite{Ketcheson2008}, a realizability-preserving second-order scheme is obtained if reconstruction and limiting are performed in every stage of the RK method.

\subsubsection{Source system}
Again writing down the finite-volume form of \eqref{eq:SplittedSystemb} yields
\begin{align*}
\dt \momentscellmean{\cellind} = \cellmean{\Source\left(\momentslocallimited{\cellind}\right)} = \Source\left(\momentscellmean{\cellind}\right)+\mathcal{O}(\dxstepsize^2).
\end{align*}
It is thus sufficient to solve the system
\begin{align}
\label{eq:ODE}
\dt \momentscellmean{\cellind} = \Source\left(\momentscellmean{\cellind}\right),
\end{align}
which is purely an ODE in every cell. For simplicity, in the following we will neglect the spatial index and the cell mean.

It can be shown, that the discontinuous-Galerkin scheme \cite{Bauer1995,Gellrich2017} applied to \eqref{eq:ODE} is unconditionally,i.e. without any time-step restriction, realizability-preserving for even polynomial orders. 

Consider the time-cell $[\timepoint{\timeind-1},\timepoint{\timeind}]$ and define $\timepoint{\timeind-\frac12} = \frac12\left(\timepoint{\timeind}-\timepoint{\timeind-1}\right)$. We consider the weak formulation of \eqref{eq:ODE}
\begin{align*}
\int_{\timepoint{\timeind-1}}^{\timepoint{\timeind}} \dt\moments\Testfunction~d\timevar -\int_{\timepoint{\timeind-1}}^{\timepoint{\timeind}} \Source\left(\moments\right)\Testfunction~d\timevar +\Testfunction(\timepoint{\timeind-1}^+)\moments(\timepoint{\timeind-1}^+)= \Testfunction(\timepoint{\timeind-1}^+)\moments(\timepoint{\timeind-1}^-),
\end{align*}
where $\moments(\timepoint{\timeind-1}^-) = \lim_{\timevar\uparrow\timepoint{\timeind-1}} \moments(\timevar)$ denotes the old solution (e.g. initial condition). This enforces continuity in a weak sense, since $\moments(\timepoint{\timeind-1}^-)\neq \moments(\timepoint{\timeind-1}^+)= \lim_{\timevar\downarrow\timepoint{\timeind-1}} \moments(\timevar)$ is possible. Transforming everything to the reference interval $[-1,1]$ and using the nodal basis
\begin{align*}
\Testfunction[0] = \frac{\timevarref^2}{2} - \frac\timevarref{2},\qquad
\Testfunction[1] = 1 - \timevarref^2,\qquad 
\Testfunction[2] = \frac{\timevarref^2}{2} + \frac\timevarref{2},
\end{align*}
yields the weak formulation
\begin{align*}
\int_{-1}^{1} \dt\momentsprojected\Testfunction[i]~d\timevarref -\frac{\dtstepsize}{2}\int_{-1}^{1} \Source\left(\momentsprojected\right)\Testfunction[i]~d\timevarref +\delta_{i0}\moments(\timepoint{\timeind-1}^+)= \Testfunction[i](-1)\moments(\timepoint{\timeind-1}^-) = \delta_{i0}\moments(\timepoint{\timeind-1}^-),
\end{align*}
$i=0,\ldots,2$, where 
\begin{align*}
\left.\moments(\timevar)\right|_{[\timepoint{\timeind-1},\timepoint{\timeind}]} = \moments(\timepoint{\timeind-1}^+)\Testfunction[0](\timevar)+\moments(\timepoint{\timeind-\frac12})\Testfunction[1](\timevar)+\moments(\timepoint{\timeind}^-)\Testfunction[2](\timevar).
\end{align*}

This is a, possibly non-linear, system for the internal values $\moments(\timepoint{\timeind-1}^+)$, $\moments(\timepoint{\timeind-\frac12})$ and $\moments(\timepoint{\timeind}^-)$, which can be solved using standard techniques. Using $\moments(\timepoint{\timeind}^-)$ we can then proceed with our next step. Note that this method is stiffly A-stable and of order $5$ but using polynomials of degree $1$, which gives order $3$, gives a timestep restriction which is only slightly better than the one obtained from explicit SSP schemes.

\subsection{Implementation details}
The source code for the numerical simulations heavily builds upon DUNE and DUNE PDELab \cite{dune-web-page}, an extensive C++ numerics framework that provides useful functionality for discretizing PDEs, e.g. interfacing with grids, parallelization, function spaces and much more. On the space and time discretized level all linear algebra operations arising from the moment system are performed with the Eigen 3 library \cite{eigenweb}. 

\section{Hyperbolicity of the Kershaw moment system}
\label{app:HyperbolicityKershaw}
Under some mild assumptions on $\FD$ the Kershaw moment system \eqref{eq:KershawP} is hyperbolic.  This is the claim of Theorem \ref{only-theorem}, which we prove below.

% \begin{thm}[Hyperbolicity of the generalized symmetric Kershaw moment system]
% 	For any distribution $F : \US \mapsto \R^+$, that is
% 	\begin{itemize}
% 		\item normalized: $\intV{\FD} = 0$ 
% 		\item symmetric w.r.t to the first moment $\intV{\FD v} = 0$
% 		\item non-flat: $\intV{\FD (x\trans v)^2} > 0, \forall x \in \US$
% 	\end{itemize}
% 	the first order moment system together with the Kershaw closure \eqref{eq:KershawP} is hyperbolic for all realizable moment vectors $(\rho,q)$, except for $|\normal{q}| = 1$ with $\normal{q}$ parallel to an eigenvector of $\intV{\FD vv\trans}$, in which case the flux Jacobian has still real eigenvalues but is no longer diagonalizable. 
% 	%\label{thm:KershawHyperbolicity}
% \end{thm}
\begin{remark}
	The original Kershaw closure with $\FD = \frac{1}{4\pi}$ and $\intV{\FD vv\trans} = I$ also loses diagonalizability for all $\normal{q}$ with $|\normal{q}| = 1$.  
\end{remark}
\begin{proof} (of Theorem \ref{only-theorem})
	The Jacobian of the fluxes in the first-order moment system is 
	\begin{align*}
		J(F\cdot n) &= 
		\begin{pmatrix}
		0 &\quad  n\trans \\
		\pfrac{P n}{\rho}  &\quad  \pfrac{P n}{q}
		\end{pmatrix} \\
		&= \begin{pmatrix}
		0 & \quad n\trans \\
		\normal{P}n - \pfrac{\normal{P}n}{\normal{q}} \normal{q}  & \quad \pfrac{\normal{P}n}{\normal{q}}
		\end{pmatrix}.
	\end{align*}
	With $\normal{P} = \normal{P}^{(K)}_{\FD}$ from \eqref{eq:KershawP} and writing $\normal{P}_{eq} = \intV{\FD vv\trans}$ we have
	\begin{align*}
	\pfrac{\normal{P}n}{\normal{q}} &= -2(\normal{P}_{eq}n)\normal{q}\trans + I (\normal{q}\trans n) + \normal{q}n\trans,\\
	\pfrac{\normal{P}n}{\normal{q}}\normal{q} &= -2 (\normal{P}_{eq}n)\normal{q}\trans\normal{q} + 2 (\normal{q}\trans n) \normal{q}.
	\end{align*}
	Inserting these into the expression for the Jacobian gives 
	\begin{align*}
	J(F \cdot n)
	&= \begin{pmatrix}
	0 & n\trans \\
	(1+|\normal{q}|^2)(\normal{P}_{eq}n) - (\normal{q}\trans n)\normal{q} &\quad -2(\normal{P}_{eq}n)\normal{q}\trans + I(\normal{q}\trans n) + \normal{q}n\trans
	\end{pmatrix},\\
	&= \begin{pmatrix}
	0 & n\trans \\
	(1+|\normal{q}|^2)(\normal{P}_{eq}n) - |\normal{q}|^2(q\trans[*] n)q\trans[*] &\quad |\normal{q}| \left( -2(\normal{P}_{eq}n)q\trans[*] + I(q\trans[*] n) + q^* n\trans \right)
	\end{pmatrix},
	\end{align*}
	where $q^* = \frac{\normal{q}}{|\normal{q}|}$ is the free-streaming first moment.
	Define the rotation matrix $\hat{R}$ that rotates $q^*$ onto the first unit vector $e_1$
	\begin{align*}
	\hat{R} q^* = e_1,
	\end{align*}
	and a compatible extension
	\begin{align*}
	R := \begin{pmatrix}
	1 & 0 \\ 0 & \hat{R}
	\end{pmatrix}
	\end{align*}
	to the full Jacobian. Under the similarity transform $R$ the Jacobian becomes
	\begin{align*}
	RJ(F\cdot n)R\trans = \begin{pmatrix}
	0 & n\trans \hat{R}\trans \\
	(1+|\normal{q}|^2)\hat{R}(\normal{P}_{eq}n) - |\normal{q}|^2 (e_1\trans \hat{R} n)e_1 &\quad |\normal{q}| \left( -2(\hat{R}\normal{P}_{eq}n)e_1\trans + I(e_1\trans \hat{R} n) + e_1 n\trans \hat{R}\trans \right)
	\end{pmatrix}
	\end{align*}
	It suffices to show hyperbolicity for arbitrary $n_j$ that form a basis of $\R^3$. Therefore we choose 
	$n_j$ such that $\hat{R}n_j = e_j, j = 1,2,3$. 
	The Jacobians become
	\begin{align*}
	RJ(F\cdot n_j)R\trans = \begin{pmatrix}
	0 & e_j\trans \\
	(1+|\normal{q}|^2)Se_j - |\normal{q}|^2 \delta_{1j} e_1 &\quad |\normal{q}| \left( -2S e_je_1\trans + I\delta_{1j} + e_1e_j\trans\right)
	\end{pmatrix},
	\end{align*}
	where $S$ is the similarity transform of $\normal{P}_{eq}$: 
	\begin{align*}
	S := \hat{R} \normal{P}_{eq} \hat{R}\trans.
	\end{align*}
	For $j=1$, we have
	\begin{align*}
	RJ(F\cdot n_1)R\trans = \begin{pmatrix}
	0 &\quad 1 &\quad 0 &\quad 0 \\
	(1+|\normal{q}|^2)S_{11} - |\normal{q}|^2 &\quad |\normal{q}|(-2S_{11} + 2) &\quad 0 &\quad 0 \\
	(1+|\normal{q}|^2)S_{21}                &\quad |\normal{q}|(-2S_{21})     &\quad |\normal{q}| &\quad 0 \\
	(1+|\normal{q}|^2)S_{31}                &\quad |\normal{q}|(-2S_{31})     &\quad 0 &\quad |\normal{q}| 
	\end{pmatrix},
	\end{align*}
	with characteristic polynomial 
	\begin{align*}
	det(RJ(F\cdot n_1)R\trans - \lambda I) &= (|\normal{q}| - \lambda)^2 \left[ \lambda^2 + \lambda ( -2 + 2 S_{11}|\normal{q}|) + (|\normal{q}|^2 - (1+|\normal{q}|^2)S_{11})\right]
	\end{align*}
	and eigenvalues 
	\begin{align*}
	\lambda_{1,2} &= |\normal{q}|, \\
	\lambda_{3,4} &= (1 - S_{11}|\normal{q}|) \pm \sqrt{g(|\normal{q}|,S_{11})}.
	\end{align*}
	We need that all eigenvalues are real, which is the case if the term under the square root
	\begin{align*}
	g(|\normal{q}|,S_{11}) &= S_{11}^2 |\normal{q}|^2 + S_{11}(|\normal{q}|-1)^2 + (1 - |\normal{q}|^2)
	\end{align*}
	is greater than or equal to zero. Since 
	\begin{align*}
	S_{11} = e_1\trans \hat{R} \intV{\FD vv\trans} \hat{R}\trans e_1 = \intV{\FD (e_1\trans \hat{R} v)^2} > 0,
	\end{align*}
	by the non-flatness assumption on $\FD$, 
	and $|\normal{q}| \in [0,1]$, we have indeed 
	\begin{align*}
	g(|\normal{q}|,S_{11}) > 0.
	\end{align*}
	We still need to check if the Jacobian is diagonalizable, i.e. there are four linear independent eigenvalues. We see immediately that the dimension of the kernel of $RJ(F\cdot n_1)R\trans - |\normal{q}| I$ is two. Thus, two independent eigenvectors exist for the eigenvalue $\lambda_{1,2} = |\normal{q}|$. Since $g > 0$, the eigenvalues $\lambda_3 \neq \lambda_4$ are distinct and therefore the Jacobian is diagonalizable. To see why we need the non-flatness assumption consider $g = 0$, which can only happen if both
	$S_{11} = 0$ and $|\normal{q}| = 1$.  In this case all four eigenvalues are equal to one and the Jacobian simplifies to 
	\begin{align*}
	RJ(F\cdot n_1)R\trans
	\begin{pmatrix}
	0 & 1 & 0 & 0 \\
	-1 & 2 & 0 & 0 \\
	(1+\normal{q}^2)S_{21} & |\normal{q}|(-2S_{21}) & 1 & 0 \\
	(1+\normal{q}^2)S_{31} & |\normal{q}|(-2S_{31})& 0 & 1
	\end{pmatrix},
	\end{align*}
	which clearly is not diagonalizable. 
	
	For $j=2$ the Jacobian is 
	\begin{align*}
	RJ(F\cdot n_2)R\trans = \begin{pmatrix}
	0 & 0 & 1 & 0 \\
	(1+|\normal{q}|^2)S_{12} & |\normal{q}|(-2S_{12}) & |\normal{q}| & 0 \\
	(1+|\normal{q}|^2)S_{22} & |\normal{q}|(-2S_{22}) & 0 & 0 \\
	(1+|\normal{q}|^2)S_{32} & |\normal{q}|(-2S_{32}) & 0 & 0 
	\end{pmatrix},
	\end{align*}
	with characteristic polynomial
	\begin{align*}
	det(RJ(F\cdot n_2)R\trans - \lambda I) &= \lambda^2 \left[ \lambda^2 + 2|\normal{q}|S_{12} \lambda + S_{22}(|\normal{q}|^2 - 1)\right]
	\end{align*}
	and eigenvalues 
	\begin{align*}
	\lambda_{1,2} &= 0, \\
	\lambda_{3,4} &= -|\normal{q}|S_{12} \pm \sqrt{h},
	\end{align*}
	where 
	\begin{align*}
	h = |\normal{q}|^2 S_{12}^2 + S_{22} (1 - |\normal{q}|^2) \geq 0
	\end{align*}
	is non-negative by the same arguments as before. Note that there is always a $\normal{q}$ for which the system is no longer hyperbolic: Choose $\normal{q}$ as a unit vector along one eigenvector of $\normal{P}_{eq}$. In that case $S$ is diagonal and $S_{12} = 0$. All eigenvalues are equal to zero and the Jacobian is
	\begin{align*}
	RJ(F\cdot n_2)R\trans = 
	\begin{pmatrix}
	0 & 0 & 1 & 0 \\
	0 & 0 & 1 & 0 \\
	2S_{22} & -2S_{22} & 0 & 0 \\
	0 & 0 & 0 & 0
	\end{pmatrix}.
	\end{align*}
	The case $j=3$ is completely analogous to $j=2$.
\end{proof}

% Bibliography
%%%%%%%%%%%%%%

\bibliographystyle{plain}
\bibliography{bibliography}

\end{document}